\documentclass{amsart}
\usepackage{graphicx}
\usepackage{tikz-cd}
\usepackage{xcolor}
\usepackage{comment}
\usepackage[T1]{fontenc}
\usepackage{xypic}

\usepackage{amsmath,amsthm,amssymb,amscd,enumerate}
\usepackage{enumitem}

\usetikzlibrary{decorations.pathmorphing}

\newtheorem{theorem}{Theorem}[section]
\newtheorem{proposition}[theorem]{Proposition}
\newtheorem{lemma}[theorem]{Lemma}
\newtheorem{corollary}[theorem]{Corollary}
\newtheorem{mthm}{Theorem}

\theoremstyle{definition}
\newtheorem{definition}[theorem]{Definition}
\newtheorem{example}[theorem]{Example}

\theoremstyle{remark}

\newtheorem{remark}[theorem]{Remark}

\numberwithin{equation}{section}

\allowdisplaybreaks[1]

\newcommand{\QQ}{\mathbb{Q}}
\newcommand{\RR}{\mathbb{R}}
\newcommand{\ZZ}{\mathbb{Z}}
\newcommand{\NN}{\mathbb{N}}

\newcommand{\QQQ}{\mathrm{Q}}

\newcommand{\rQQQ}{\widehat{\mathrm{Q}}}
\newcommand{\HHH}{\mathrm{H}}
\newcommand{\CCC}{\mathrm{C}}
\newcommand{\ZZZ}{\mathrm{Z}}
\newcommand{\BBB}{\mathrm{B}}

\DeclareMathOperator{\cl}{\mathrm{cl}}
\DeclareMathOperator{\scl}{\mathrm{scl}}

\DeclareMathOperator{\Ker}{\mathrm{Ker}}
\DeclareMathOperator{\Coker}{\mathrm{Coker}}

\newcommand{\Gg}{G}
\newcommand{\Ng}{N}

\newcommand{\Hg}{H}
\newcommand{\Kg}{K}
\newcommand{\Gam}{\Gamma}

\newcommand{\CG}{[\Gg,\Gg]}
\newcommand{\CGN}{[\Gg,\Ng]}

\newcommand{\sclG}{\scl_{\Gg}}

\newcommand{\sclGN}{\scl_{\Gg,\Ng}}

\newcommand{\barsclG}{\overline{\scl}_{\Gg}}
\newcommand{\barsclGN}{\overline{\scl}_{\Gg,\Ng}}

\newcommand{\xl}{x}
\newcommand{\yl}{y}

\newcommand{\gl}{g}
\newcommand{\hl}{h}

\newcommand{\xbr}{\check{x}}

\newcommand{\nuf}{\nu}
\newcommand{\muf}{\mu}
\newcommand{\psf}{\psi}
\newcommand{\phf}{\phi}

\newcommand{\ppi}{p}

\newcommand{\QNG}{\QQQ(\Ng)^{\Gg}}

\newcommand{\WW}{\mathrm{W}}
\newcommand{\WGN}{\WW(\Gg,\Ng)}

\newcommand{\DD}{D}

\newcommand{\genus}{g}
\newcommand{\Sg}{\Sigma_{\genus}}

\newcommand{\ChGN}{\mathcal{C}^{\mathrm{h}}(\Gg,\Ng)}
\newcommand{\ChG}{\mathcal{C}^{\mathrm{h}}(\Gg)}
\newcommand{\barChGN}{\overline{\mathcal{C}}^{\mathrm{h}}(\Gg,\Ng)}
\newcommand{\barChG}{\overline{\mathcal{C}}^{\mathrm{h}}(\Gg)}

\newcommand{\Rdim}{\dim_{\mathbb{R}}}
\newcommand{\op}{\mathrm{op}}
\newcommand{\hhh}{\mathrm{h}}

\newcommand{\CR}{\mathcal{C}_{\mathbb{R}}}
\newcommand{\CQ}{\mathcal{C}_{\mathbb{Q}}}
\newcommand{\CZ}{\mathcal{C}_{\mathbb{Z}}}

\makeatletter
\@namedef{subjclassname@2020}{%
  \textup{2020} Mathematics Subject Classification}
\makeatother


\title[Coarse geometry of mixed scl I]{Coarse geometry of stable mixed commutator length  I:  duality and  functional analysis on chains}

\author[M. Kawasaki]{Morimichi Kawasaki}
\address[Morimichi Kawasaki]{Department of Mathematics, Faculty of Science, Hokkaido University, North 10, West 8, Kita-ku, Sapporo, Hokkaido 060-0810, Japan.}
\email{kawasaki@math.sci.hokudai.ac.jp}

\author[M. Kimura]{Mitsuaki Kimura}
\address[Mitsuaki Kimura]{Department of Mathematics, Osaka Dental University
 8-1 Kuzuha-hanazono-cho, Hirakata, Osaka 573-1121 Japan}
\email{kimura-m@cc.osaka-dent.ac.jp}

\author[S. Maruyama]{Shuhei Maruyama}
\address[Shuhei Maruyama]{ School of Mathematics and Physics, College of Science and Engineering, Kanazawa University, Kakuma-machi, Kanazawa, Ishikawa, 920-1192, Japan}
\email{smaruyama@se.kanazawa-u.ac.jp}

\author[T. Matsushita]{Takahiro Matsushita}
\address[Takahiro Matsushita]{Department of Mathematical Sciences, Shinshu University, Matsumoto, Nagano 390-8621, Japan}
\email{matsushita@shinshu-u.ac.jp}

\author[M. Mimura]{Masato Mimura}
\address[Masato Mimura]{Mathematical Institute, Tohoku University, 6-3, Aramaki Aza-Aoba, Aoba-ku, Sendai 980-8578, Japan}
\email{m.masato.mimura.m@tohoku.ac.jp}

\begin{document}

\begin{abstract}
Let $G$ be a group and $N$ its normal subgroup. On the mixed commutator subgroup $[G,N]$, the mixed stable commutator length $\mathrm{scl}_{G,N}$ and the restriction of the ordinary stable commutator length $\mathrm{scl}_{G}$ are defined. We characterize when they are bi-Lipschitz equivalent by the vanishing of  a certain $\RR$-linear space $\mathrm{W}(G,N)$ related to invariant quasimorphisms. For the proof, we obtain a refined version of the generalized mixed Bavard duality theorem, and perform functional analysis on the completion of a certain space of $1$-chains.
\end{abstract}

\maketitle

\section{Introduction: bi-Lipschitz comparison problem}
\label{section=Intro}
\subsection{Main result: bi-Lipschitz comparison in full generality}\label{subsec=Intro}
The main theme of the present work is the \emph{large-scale behavior} of the stable mixed commutator length $\scl_{\Gg,\Ng}$ (``mixed scl'') on the mixed commutator subgroup $[\Gg,\Ng]$. More precisely, we are interested in the comparison of it with the ordinary stable commutator length $\scl_{\Gg}$ (``scl'') on the ordinary commutator subgroup $[\Gg,\Gg]$, restricted to $[\Gg,\Ng]$. In this Part I paper, we build the general theory for a pair $(\Gg,\Ng)$ of a group $\Gg$ and its normal subgroup $\Ng$ in full generality; in Part II paper, we will discuss the special case where $\Gam=\Gg/\Ng$ is  nilpotent  in more detail.

Throughout this introduction, let $\Gg$ be a group and $\Ng$ its normal subgroup. Let $\mathcal{S}_{\Gg,\Ng}=\{[\gl,\xl]=\gl\xl\gl^{-1}\xl^{-1}\,|\,\gl\in \Gg,\xl\in \Ng\}$ be the set of simple $(\Gg,\Ng)$-commutators. The \emph{mixed commutator subgroup} $[\Gg,\Ng]$ is the group generated by $\mathcal{S}_{\Gg,\Ng}$; the \emph{mixed commutator length} $\cl_{\Gg,\Ng}\colon [\Gg,\Ng]\to \ZZ_{\geq 0}$ is the word length with respect to $\mathcal{S}_{\Gg,\Ng}$, namely, for $\yl\in [\Gg,\Ng]$ we define
\[
\cl_{\Gg,\Ng}(\yl)=\min \{n\in \ZZ_{\geq 0}\;|\; \textrm{there exist }\yl_1,\ldots,\yl_n\in \mathcal{S}_{\Gg,\Ng}\textrm{ such that }\yl=\yl_1\cdots \yl_n\}.
\]
\begin{definition}[$\sclGN$]\label{defn=mixedscl}
The \emph{stable mixed commutator length} $\sclGN$ is defined by 
\[
\scl_{\Gg,\Ng}\colon [\Gg,\Ng]\to \RR_{\geq 0};\quad \scl_{\Gg,\Ng}(\yl)=\lim_{k\to \infty}\frac{\cl_{\Gg,\Ng}(\yl^k)}{k}.
\]
\end{definition}
For the case of $\Ng=\Gg$, $\cl_{\Gg,\Gg}$ and $\scl_{\Gg,\Gg}$ coincide with the (ordinary) commutator length $\cl_{\Gg}$ and stable commutator length $\scl_{\Gg}$, respectively. We refer the reader to \cite{Calegari} for backgrounds of $\scl_{\Gg}$, and to \cite{KKMMMsurvey} for a survey on $\scl_{\Gg,\Ng}$. Here, we emphasize that $\scl_{\Gg,\Ng}$ has a strong connection to \emph{invariant quasimorphisms} by the mixed Bavard duality theorem
\begin{equation}\label{eq=mixedBavard}
\scl_{\Gg,\Ng}(\yl)=\sup_{[\muf]\in (\QQQ(\Ng)^{\Gg}/\HHH^1(\Ng)^{\Gg})\setminus \{0\}}\frac{|\muf(\yl)|}{2\DD(\muf)};
\end{equation}
see Subsection~\ref{subsec=Bavard} for more details. 
Here, $\QQQ(\Ng)^{\Gg}$ denotes the space of homogeneous $\Gg$-invariant quasimorphisms on $\Ng$; as is recalled in the next subsection.

In this introduction, we present our first main result on the \emph{bi-Lipschitz comparison problem} between the ``mixed scl'' $\scl_{\Gg,\Ng}$ and the ``scl'' $\scl_{\Gg}$ on $[\Gg,\Ng]$, which asks when there exists  $C\in \RR_{\geq 1}$  such that for every $\yl\in [\Gg,\Ng]$,
\begin{equation}\label{eq=biLipscl}
\scl_{\Gg,\Ng}(\yl)\leq C\cdot \scl_{\Gg}(\yl)
\end{equation}
holds. 
We note that the inequality  $\scl_{\Gg,\Ng}(\yl)\geq \scl_{\Gg}(\yl)$ in the opposite direction always holds since  $\mathcal{S}_{\Gg,\Ng}\subseteq \mathcal{S}_{\Gg,\Gg}$. 
Now we present our first main result, which solves this bi-Lipschitz comparison problem in the full generality. 

\begin{mthm}[bi-Lipschitz comparison in the full generality]\label{mthm=main_biLip}
Let $\Gg$ be  a group and $\Ng$ its normal subgroup. Then, the following two conditions are equivalent.
\begin{enumerate}[label=\textup{(\roman*)}]
  \item $\scl_{\Gg}$ and $\scl_{\Gg,\Ng}$ are bi-Lipschitz equivalent on $[\Gg,\Ng]$, namely, there exists  $C\in \RR_{\geq 1}$  such that for every $\yl\in [\Gg,\Ng]$ we have \eqref{eq=biLipscl}.
  \item $\WW(\Gg,\Ng)=0$.
\end{enumerate}
Here, $\WW(\Gg,\Ng)$ is \relax a quotient $\RR$-linear space defined in Definition~\textup{\ref{defn=W}}.
\end{mthm}
\relax Roughly, $\WW(\Gg,\Ng)$ is the space of non-extendable invariant quasimorphisms modulo invariant homomorphisms. On Theorem~\ref{mthm=main_biLip}, let us emphasize the following three points; we will discuss more details in Remark~\ref{rem=dimW} and Subsection~\ref{subsec=previous}.

\begin{enumerate}[label=(\arabic*)]
  \item The direction of ``(ii) implies (i)'' was proved in the previous work \cite{KKMMM}. Hence, the \emph{novel part of Theorem~\textup{\ref{mthm=main_biLip}} is the direction of ``\textup{(i)} implies \textup{(ii)}.''} 
  \item The space $\WW(\Gg,\Ng)$ may be understood better than mixed scl. The background here is that $\WW(\Gg,\Ng)$ is related to (ordinary and bounded)  group cohomology, and hence that the homological algebraic method can be applied. From this viewpoint, Theorem~\ref{mthm=main_biLip} may be regarded as a result of providing information of (mixed) scl from that of $\WGN$.
  \item In (1), ``(i) implies (ii)'' is equivalent to saying that ``not (ii) implies not (i).'' From the perspective of (2), to show this novel direction  of Theorem~\ref{mthm=main_biLip} we are required to construct a witness of non-bi-Lipschitz-equivalence between $\scl_{\Gg}$ and $\scl_{\Gg,\Ng}$ on $\CGN$, provided that $\WGN\ne 0$. In \cite{KK} and \cite{MMM}, constructions of such witnesses were done for specific pairs of $(\Gg,\Ng)$, respectively. However, in both work, they were able to do so because they had obtained  an explicit representative of a non-zero element in $\WGN$. 

The challenge in the present work, to prove Theorem~\ref{mthm=main_biLip},  is to construct a witness of ``not (i)'' under the assumption of $\WGN\ne 0$, where \emph{no} further piece of information on a non-zero element of $\WGN$ is given.\end{enumerate}

Additionally, we remark  that the QI (quasi-isometry) comparison of scl reduces to the bi-Lipschitz comparison above due to the following semi-homogeneity: for every $\yl\in [\Gg,\Ng]$ and for every $k\in \ZZ$, $\scl_{\Gg,\Ng}(\yl^k)=|k|\cdot \scl_{\Gg,\Ng}(\yl)$.

\subsection{Invariant quasimorphisms and $\WGN$}\label{subsec=W}
We will present a detailed definition of  the $\RR$-linear  space $\WW(\Gg,\Ng)$, which appears in the statement of Theorem~\ref{mthm=main_biLip};  this space is related to homogeneous \emph{$\Gg$-invariant quasimorphisms} on $\Ng$. 
\begin{definition}\label{defn=qm}
Let $\Gg$ be a group and $\Ng$ its normal subgroup.
\begin{enumerate}[label=(\arabic*)]
  \item For a function $f\colon\Gg\to\RR$,  the \emph{defect} of $f$ is defined by
\begin{equation*}\label{eq=defect}
 \DD(f)=\sup\left\{\left|f(\gl_1\gl_2)-f(\gl_1)-f(\gl_2)\right| \,\middle|\, \gl_1,\gl_2\in \Gg\right\} \in \RR_{\geq 0}\cup\{\infty\}. 
\end{equation*}
  \item A function $\psf\colon \Gg\to \RR$ is called a \emph{quasimorphism} on $\Gg$ if $\DD(\psf)<\infty$.
  \item A function from $\Gg$ to $\RR$ is said to be \emph{homogeneous} if its restriction to every cyclic subgroup is a homomorphism. 
  \item The symbol $\QQQ(\Gg)$ denotes the $\RR$-linear space of homogeneous quasimorphisms on $\Gg$. 
\end{enumerate}
\end{definition}
We note that $\QQQ(\Gg)\supseteq \HHH^1(\Gg)$, where  we identify the first group cohomology $\HHH^1(\Gg)=\HHH^1(\Gg;\RR)$ with  the $\RR$-linear space $\mathrm{Hom}(\Gg,\RR)$ of group homomorphisms from $\Gg$ to $\RR$. For the pair $(\Gg,\Ng)$, we equip the space of $\RR$-valued functions on $\Ng$ with the adjoint action of $\Gg$, and define $\QQQ(\Ng)^{\Gg}$ and $\HHH^1(\Ng)^{\Gg}$ as the $\Gg$-invariant parts of $\QQQ(\Ng)$ and $\HHH^1(\Ng)$; the concrete definition goes as follows.

\begin{definition}\label{defn=QNG}
Let $\Gg$ be a group and $\Ng$ its normal subgroup. Then, define 
\[
\QQQ(\Ng)^{\Gg}=\{\muf \in \QQQ(\Ng)\;|\; \muf(\gl\xl\gl^{-1})=\muf({\xl})
\mathrm{\ for\ every\ }\gl\in \Gg\ \mathrm{and\ every\ }\xl\in \Ng\}.
\]
\end{definition}
Because $\QQQ(\Gg)=\QQQ(\Gg)^{\Gg}$ (Lemma~\ref{lem=quasiinv}), the inclusion map $i\colon \Ng\hookrightarrow \Gg$ induces an $\RR$-linear map $i^{\ast}\colon \QQQ(\Gg)=\QQQ(\Gg)^{\Gg}\to \QQQ(\Ng)^{\Gg}$;  $\phi \mapsto \phi|_{\Ng}$. Now, \relax we define the $\RR$-linear space $\WW(\Gg,\Ng)$ as follows.

\begin{definition}[\relax $\WGN$]\label{defn=W}
Let $\Gg$ be a group and $\Ng$ its normal subgroup. Then, we define the quotient  space $\WGN$ of $\QQQ(\Ng)^{\Gg}$ by
\[
\WW(\Gg,\Ng)=\QQQ(\Ng)^{\Gg}/\left(\HHH^1(\Ng)^{\Gg}+i^{\ast}\QQQ(\Gg)\right).
\] 
\end{definition}

Roughly speaking,  homogeneous invariant quasimorphisms that are \emph{extendable to $\Gg$} (elements in $i^{\ast}\QQQ(\Gg)$) and  invariant homomorphisms (elements in $\HHH^1(\Ng)^{\Gg}$)  are both `uninteresting' invariant quasimorphisms: for the  former  ones, due to Lemma~\ref{lem=quasiinv} there is no mystery for the $\Gg$-invariance. At the level of $\WGN$, exactly linear combinations of them represent the zero element.

\begin{remark}\label{rem=dimW}
\relax As we mentioned in Subsection~\ref{subsec=Intro}, the homological algebraic method may be used to study $\WGN$. As an outcome, the dimension of $\WW(\Gg,\Ng)$ is  computable for several pairs $(\Gg,\Ng)$ as follows. Assume that $\Gam=\Gg/\Ng$ is amenable, or, more generally that $\HHH^2_b(\Gam)=\HHH^3_b(\Gam)=0$, where $\HHH^{\bullet}_b$ denotes bounded cohomology (as we recall in Subsection~\ref{subsec=H_b}). Then, it was proved in \cite[Theorem~1.10]{KKMMM} that 
\[
\WW(\Gg,\Ng)\cong \mathrm{Im}(c^2_{\Gg})\cap \mathrm{Im}(\ppi^{\ast}\colon \HHH^2(\Gam)\to \HHH^2(\Gg))
\]
holds, where $c^2_{\Gg}\colon \HHH^2_b(\Gg)\to \HHH^2(\Gg)$ is the second comparison map, and $\ppi\colon \Gg\twoheadrightarrow \Gam$ is the group quotient map. In particular,  if $\HHH^2_b(\Gam)=\HHH^3_b(\Gam)=0$,   then we have
\[
\Rdim \WGN\leq \min \{\Rdim\HHH^2(\Gam),\Rdim\HHH^2(\Gg)\}.
\]

For instance, for the pair $(\Gg,\Ng)$ with $\Gg=\pi_1(\Sg)$ and $\Ng=[\Gg,\Gg]$ for $\genus$ at least two, it was seen in \cite[Theorem~1.1]{KKMMM} that
\begin{equation}\label{eq=Wsurface}
\WW(\Gg,\Ng)\cong \HHH^2(\Gg)\cong \HHH^2(\Sg;\RR)= \RR.
\end{equation}
Here, $\Sg$ is a connected orientable closed  surface of genus $\genus$.
\end{remark}

\subsection{Relation to the previous work}\label{subsec=previous}
We review previous work on the bi-Lipschitz comparison problem. 

\begin{enumerate}[label=(\arabic*)]
  \item In \cite{KKMMM}, the authors proved the following.
\begin{theorem}[{\cite[Theorem~2.1]{KKMMM}}]\label{prop=(ii)implies(i)}
Assume that $\WW(\Gg,\Ng)=0$. Let $\Gam=\Gg/\Ng$. Then, the following hold.
\begin{enumerate}[label=\textup{(\arabic*)}]
  \item $\scl_{\Gg}$ and $\scl_{\Gg,\Ng}$ are bi-Lipschitz equivalent on $[\Gg,\Ng]$.
  \item If $\Gam$ is solvable, then $\scl_{\Gg}\equiv \scl_{\Gg,\Ng}$ on $[\Gg,\Ng]$.
  \item If $\Gam$ is amenable, then for every $\yl\in [\Gg,\Ng]$ we have
\[
\scl_{\Gg,\Ng}(\yl)\leq 2\scl_{\Gg}(\yl).
\]
\end{enumerate}
\end{theorem}
 Theorem~\ref{prop=(ii)implies(i)}~(1) verifies the direction of ``(ii) implies (i)'' in Theorem~\ref{mthm=main_biLip}.
 \item \relax As is discussed in Subsection~\ref{subsec=Intro}, in the proof of (the novel part of) Theorem~\ref{mthm=main_biLip} we are required to construct a witness of ``not (i)'' under the assumption of ``not (ii)'' ($\WGN\ne 0$). The first construction of this sort was given by \cite{KK} in the case where $\Gg=\mathrm{Symp}_0(\Sg,\omega)$ and $\Ng=\mathrm{Ham}(\Sg,\omega)$ for $\genus$ at least two; see \cite{KK} for details. Later, in \cite[Example~7.15]{KKMMM}, we obtained a refined example, with smaller $\Gg$ than the one above. For both of the consructions in \cite{KK} and \cite{KKMMM}, Py's Calabi quasimorphism, defined in \cite{Py06}, was employed  as a representative of a non-zero element in $\WGN$.
\item \relax Another construction of a witness of ``not (ii)'' was given in \cite{MMM} for $\Gg=\pi_1(\Sg)$ and $\Ng=[\Gg,\Gg]$, where $\genus$ is at least two. In this example, $\Rdim \WW(\Gg,\Ng)=1$ by \eqref{eq=Wsurface}. In this work \cite{MMM}, they obtained a representative of a non-zero element in $\WGN$ from a circle action of $\Gg$, and employed it.
\end{enumerate}

\subsection{Strategy of the proof of the main result}\label{subsec=strategy}
\relax To summarize our discussion in Subsection~\ref{subsec=Intro} on Theorem~\ref{mthm=main_biLip}, given $(\Gg,\Ng)$ with $\WGN\ne 0$, we are required to construct a witness of ``not (i),'' that is, a sequence $(\yl_n)_{n\in \NN}$ in $\CGN$ satisfying
\begin{equation}\label{eq=ratio_infty}
\lim_{n\to \infty}\frac{\scl_{\Gg,\Ng}(\yl_n)}{\scl_{\Gg}(\yl_n)}\to \infty.
\end{equation}
The mixed Bavard duality theorem \eqref{eq=mixedBavard}, which was proved in \cite{KKMM1} in the full generality, relates $\scl_{\Gg,\Ng}$ to the normed space $(\QQQ(\Ng)^{\Gg}/\HHH^1(\Ng)^{\Gg},D)$. Here, the defect $\DD$ on $\QQQ(\Ng)$ gives rise to a genuine norm on the quotient $\RR$-linear space $\QQQ(\Ng)^{\Gg}/\HHH^1(\Ng)^{\Gg}$. The original case for $\Ng=\Gg$ is known as the celebrated Bavard duality theorem \cite{Bavard}.  For the case of $\Ng=\Gg$, in  \cite{Calegari} and \cite{Calegari09} Calegari extended the framework of the Bavard duality theorem from the commutator subgroup $[\Gg,\Gg]$ to a certain subspace of real $1$-chains of $\Gg$ by establishing the generalized Bavard duality theorem. In \cite{KKMMM_KIAS}, we obtained the mixed version, the generalized mixed Bavard duality theorem. In this manner, despite the fact that $\CGN$ is not equipped with a structure of an $\RR$-linear space in general, we switch our setting from $\CGN$ to the space of $1$-chains above; this is an $\RR$-linear space, as well as  the space $\QQQ(\Ng)^{\Gg}/\HHH^1(\Ng)^{\Gg}$. This is crucial for us to formulate the  refined version of the generalized mixed  Bavard duality theorem, which we will discuss below.

Our strategy is \emph{functional analytic}, as we will describe in this paragraph. In the first step of the proof of Theorem~\ref{mthm=main_biLip}, we refine the generalized mixed  Bavard duality theorem to fit into the framework of functional analysis, which we state as Theorem~\ref{thm=refinedgmBavard}. More precisely, the refinement should be of the following form: for a certain real normed space coming from $1$-chains of $\Ng$, the continuous dual of this space is isometrically isomorphic to $(\QQQ(\Ng)^{\Gg}/\HHH^1(\Ng)^{\Gg},D)$ (such a refinement will make the meaning of the \emph{duality} theorem even clearer). In the second step, we perform functional analysis with the aid of this refined version. Here, the key idea is to \emph{take the completion} of the real normed space. Functional analysis works effectively on this completion. However, here we emphasize that the Banach space, resulting in this completion procedure, is quite mysterious and \emph{care is needed to handle this space} in general; see Remark~\ref{rem=ghost} and Example~\ref{exa=iotakernel}.

Our final step is to get back from this mysterious Banach space to the mixed commutator subgroup $\CGN$ \relax by approximation; compare with Lemma~\ref{lem=CGNchain} and Remark~\ref{rem=coarsehom}. \relax At the moment, a scaling-up process is involved in our approximation (as we see in Lemma~\ref{lem=approximation_scaling}), but we are  still able to obtain a sequence $(\yl_n)_{n\in \NN}$ with \eqref{eq=ratio_infty}. This ends the sketch of our proof of Theorem~\ref{mthm=main_biLip}.

\subsection{Further direction: coarse kernels}\label{subsec=further}
The strategy of the proof in Subsection~\ref{subsec=strategy} indicates that $\scl_{\Gg}$ and $\scl_{\Gg,\Ng}$ are quite different on $[\Gg,\Ng]$ as long as $\WGN\ne 0$. However, from the viewpoint of large-scale geometry, the most extreme situation is that there exists a subset $A$ of $[\Gg,\Ng]$ such that
\begin{equation}\label{eq=bddunbdd}
\textrm{$A$ is \emph{bounded} in $\scl_{\Gg}$,}\quad \textrm{but}\quad \textrm{$A$ is \emph{unbounded} in $\scl_{\Gg,\Ng}$}.
\end{equation}
Indeed, if such $A$ exists, then we can extract a sequence $(\yl_n)_{n\in \NN}$ in $A$ such that  $(\scl_{\Gg,\Ng}(\yl_n))_n$ diverges whereas $(\scl_{\Gg}(\yl_n))_n$ stays bounded, thus supplying the strongest form of \eqref{eq=ratio_infty}. The `largest' subset $A\subseteq [\Gg,\Ng]$ with \eqref{eq=bddunbdd}, in the sense of coarse geometry, may be expressed as the (non-trivial) \emph{coarse kernel} in the category of \emph{coarse groups} \cite{LV}. We will briefly describe the theory of coarse groups in the setting of the present paper in Subsection~\ref{subsec=coarse}; the definition of coarse kernel in this setting is given as  Definition~\ref{defn=coarsekernel} there. 

Theorem~\ref{mthm=main_biLip} only treats non-bi-Lipschitz-equivalence, and it does not reach the study of coarse kernels. Nevertheless, we have our  second main theorem, Theorem~\ref{mthm=main_Qcoarsekernel}, in the setting of rational chains. We will discuss this topic in Section~\ref{section=coarsekernel}. We additionally remark that a coarse kernel has a close relation to the extendability of invariant quasimorphisms up to invariant homomorphisms, as we discuss in Theorem~\ref{thm=cCKextend}. As we hinted at the beginning of this introduction, in the case where $\Gam=\Gg/\Ng$ is nilpotent (in particular abelian), we are able to determine the coarse kernel at the level of group elements: we will discuss this in our Part II paper \cite{KKMMM_partII}.

\subsection*{Organization of the present paper}
In Section~2, we exhibit Theorem~\ref{mthm=main_Qcoarsekernel}, which treats a coarse kernel in the setting of rational chains. Section~\ref{sec=prelim} is for preliminaries. Section~\ref{sec=refinedBavard} is devoted to the refined version of the generalized mixed Bavard duality theorem (Theorem~\ref{thm=refinedgmBavard}). In Section~\ref{sec=proofmain}, we prove Theorem~\ref{mthm=main_biLip} and Theorem~\ref{mthm=main_Qcoarsekernel}. We also show Theorem~\ref{thm=cCKextend} there. In Section~\ref{sec=applications}, we discuss applications. In Section~\ref{sec=closedrange}, we obtain a criterion for $\WGN$ being a normed space (Theorem~\ref{thm=closedrange}).

\subsection*{Notation and conventions}
Let $\NN=\{n\in \ZZ\;|\;n>0\}$. We write $\DD_{\Hg}(\psf)$, rather than $\DD(\psf)$, for the defect of a quasimorphism $\psf$ on a group $\Hg$ if we indicate the group $\Hg$. For ordinary  and  bounded cohomology of groups, the coefficient group is assumed to be $\RR$ if it is omitted. For instance, $\HHH^n(\Gg)$ means $\HHH^n(\Gg;\RR).$ For a quotient set $Y=X/\sim$, we write $[x]$ for the equivalence class represented by $x\in X$, as an element of $Y$. In the statements of the Bavard duality theorem and its strengthenings (such as \eqref{eq=mixedBavard}; see Subsection~\ref{subsec=Bavard} for more details), we define the supremum on the right-hand side  to be zero if this is over the empty set. For a real $1$-chain $c$ of a group $\Gg$ and a function $f\colon \Gg \to \RR$, we define the value $f(c)$ by linear extension. 
For $a,b\in \RR$ and $C\in \RR_{\geq 0}$, we write $a\sim_C b$ if $|a-b|\leq C$. For an $\RR$-linear space $V$, $\Rdim V$ means the real dimension of $V$: if $V$ is infinite-dimensional, then we set $\Rdim V=\infty$. For a metric space $(X, d_X)$, $x\in X$ and $R\in \RR_{\geq 0}$, we write $B^{(X,d_X)}_R(x)$ for the closed ball of radius $R$ centered at $x$.

\section{Main Result on coarse kernels for rational chains}\label{section=coarsekernel}
\relax As an extended introduction, in this section we present our second main theorem, Theorem~\ref{mthm=main_Qcoarsekernel}. This theorem treats the coarse kernel of a map between certain spaces of rational $1$-chains equipped with (mixed) scl seminorms. Hence, before stating Theorem~\ref{mthm=main_Qcoarsekernel}, we briefly describe our motivation for studying large-scale behavior of mixed scl, as well as  recalling the definition of mixed scl for certain rational $1$-chains. 
Throughout this section, let $\Gg$ be a group and $\Ng$ its normal subgroup. Let $i\colon \Ng\hookrightarrow \Gg$ be the inclusion map. 

\subsection{Motivation for the study of coarse groups}\label{subsec=coarsekernel}
First, we describe our motivation to study the large-scale geometry of mixed scl. 
The point here is that $\scl_{\Gg,\Ng}$ is \emph{not} a length function, but the following weak form of the triangle inequality holds: for all $\yl_1,\yl_2\in [\Gg,\Ng]$, 
\[
\scl_{\Gg,\Ng}(\yl_1\yl_2)\leq \scl_{\Gg,\Ng}(\yl_1)+\scl_{\Gg,\Ng}(\yl_2)+\frac{1}{2}.
\]
(To see this, apply the mixed Bavard duality theorem \eqref{eq=mixedBavard}, for instance.) We extend the domain of $\scl_{\Gg,\Ng}$ and define
\begin{align*}
&\scl_{\Gg,\Ng}\colon \Gg\to \RR_{\geq 0}\cup \{\infty\};\\
&\gl\mapsto 
\begin{cases} \dfrac{\scl_{\Gg,\Ng}(\gl^k)}{k}, & \text{if $\gl\in \Ng$ and there exists $k\in \NN$ with $\gl^k\in [\Gg,\Ng]$},\\
\infty, &\mathrm{otherwise}.
\end{cases}
\end{align*}
We note that this formulation is consistent with the definition of $\scl_{\Gg,\Ng}$ on $\QQ$-chains; compare with Definition~\ref{defn=sclQchain} in Subsection~\ref{subsec=Qchains}. Now define $d_{\scl_{\Gg,\Ng}}\colon \Gg\times \Gg\to \RR_{\geq 0}\cup \{\infty\}$ and $d_{\scl_{\Gg,\Ng}}^+\colon \Gg\times \Gg\to \RR_{\geq 0}\cup \{\infty\}$ by
\[
d_{\scl_{\Gg,\Ng}}(\gl_1,\gl_2)=\scl_{\Gg,\Ng}(\gl_1^{-1}\gl_2)
\]
and
\begin{equation}\label{eq=d+scl}
d^+_{\scl_{\Gg,\Ng}}(\gl_1,\gl_2)=\begin{cases}
\scl_{\Gg,\Ng}(\gl_1^{-1}\gl_2)+\dfrac{1}{2}, &\mathrm{if\ }\gl_1\ne \gl_2,\\
0, &\mathrm{if\ }\gl_1= \gl_2,
\end{cases}
\end{equation}
for $\gl_1,\gl_2\in \Gg$. 
 We set $d_{\scl_{\Gg}}=  d_{\scl_{\Gg,\Gg}}$ and $d^+_{\scl_{\Gg}}=  d^+_{\scl_{\Gg,\Gg}}$. 
Then, whereas $d_{\scl_{\Gg,\Ng}}$ may fail the triangle inequality, $d^+_{\scl_{\Gg,\Ng}}$ gives rise to  a genuine metric on $\Gg$, possibly taking the value $\infty$. The small-scale behaviors of  $d_{\scl_{\Gg,\Ng}}$ and $d^+_{\scl_{\Gg,\Ng}}$ may be completely different; for instance, the topology given by the metric $d^+_{\scl_{\Gg,\Ng}}$ is always discrete. On the other hand, the large-scale behaviors of $d_{\scl_{\Gg,\Ng}}$ and $d^+_{\scl_{\Gg,\Ng}}$ are exactly the same. This is our background to study the large-scale geometry of (mixed) scl. The group $\Gg$ admits two metrics (possibly taking the value $\infty$) $d^+_{\scl_{\Gg}}$ and $d^+_{\scl_{\Gg,\Ng}}$, and both of them are  bi-invariant. 
Here, a map $d \colon \Gg \times \Gg \to \RR_{\geq 0}\cup \{\infty\}$ (such as a metric) is said to be \emph{bi-invariant} if for all $g ,g_1,g_2, g' \in \Gg$, $d(gg_1g', gg_2 g')=d(g_1,g_2)$ holds. 
Thus, comparison problems between $\scl_{\Gg}$ and $\scl_{\Gg,\Ng}$ in  the  large-scale setting might be of interest; we refer the reader to \cite{BIP} for the study of bi-invariant metrics on groups of geometric origin. We will go back to this point in Remark~\ref{rem=coarsehom}.

\subsection{Mixed scl on rational chains}\label{subsec=Qchains}
Here we briefly describe the setting of the generalized mixed Bavard duality theorem: the definition of the subspace of rational $1$-chains of $\Ng$ appearing in the theorem and of mixed scl on this rational chain space. Let $R$ be a unital commutative ring. For $k\in \NN$, let $\CCC_k(\Gg;R)$ be the group of $k$-chains of $\Gg$ with $R$-coefficients. Define $\CCC_2'(\Gg,\Ng;R)$ to be the $R$-submodule of $\CCC_2(\Gg;R)$ generated by the set
\[
\{(\gl_1,\gl_2)\in \CCC_2(\Gg;R)\;|\; \mathrm{either\ }\gl_1\mathrm{\ or\ }\gl_2\mathrm{\ belongs\ to\ }\Ng\}.
\]
Recall that the boundary map $\partial\colon \CCC_2(\Gg;R)\to \CCC_1(\Gg;R)$ is given by
\[
\partial ((\gl_1,\gl_2))=\gl_2-\gl_1\gl_2+\gl_1.
\]
\begin{definition}\label{defn=chainspace}
In the setting above, we define the following space of chains.
\begin{enumerate}[label=(\arabic*)]
  \item Define $\BBB'_1(\Gg,\Ng;R)$ to be $\partial \CCC'_2(\Gg,\Ng;R)$.
  \item Define $\mathcal{C}_R(\Gg, \Ng)$ by
\[
\mathcal{C}_R(\Gg,\Ng)=\CCC_1(\Ng;R)\cap \BBB'_1(\Gg,\Ng;R).
\]
\end{enumerate}
We also set $\BBB_1(\Gg;R)=\BBB'_1(\Gg,\Gg;R)$ and $\mathcal{C}_R(\Gg)=\mathcal{C}_R(\Gg,\Gg) (=\CCC_1(\Gg;R)\cap \BBB_1(\Gg;R))$.
\end{definition}

In the present work, we will focus on the case where $R=\RR$ to perform functional analysis on this subspace of $\RR$-chains. Nevertheless, here we only define the mixed scl on the rational part, $\CQ(\Gg,\Ng)$, because  its  definition on $\QQ$-chains is much easier to explain. We also note that
\[
[\Gg,\Ng]\subseteq \CZ(\Gg,\Ng)\subseteq \CQ(\Gg,\Ng)\subseteq \CR(\Gg,\Ng),
\]
where we regard an element in $\Ng$ as a $1$-chain; the first inclusion follows from Lemma~\ref{lem=CGNchain}. For $c\in \CQ(\Gg,\Ng)$, there exists $k\in \NN$ such that $kc\in \CZ(\Gg,\Ng)$. An element $c$ in $\CZ(\Gg,\Ng)\subseteq \CCC_1(\Ng;\ZZ)$ is expressed as
\begin{equation}\label{eq=Zchain}
c=\xl_1+\cdots +\xl_m- \check{\xl}_1-\cdots -\check{\xl}_n,
\end{equation}
where $m,n\in \ZZ_{\geq 0}$ and $\xl_1,\ldots,\xl_m,\check{\xl}_1,\ldots,\check{\xl}_n\in \Ng$. The following definitions are known to be well-defined, and $\scl_{\Gg,\Ng}$ gives rise to a seminorm on the $\QQ$-linear space $\mathcal{C}_{\QQ}(\Gg,\Ng)$; we refer the reader to \cite{KKMMM_KIAS} for these proofs.

\begin{definition}[mixed scl for $\QQ$-chains]\label{defn=sclQchain}
\begin{enumerate}[label=(\arabic*)]
  \item Let $c\in \CZ(\Gg,\Ng)$ of the form \eqref{eq=Zchain}. We define the \emph{mixed commutator length} $\cl_{\Gg,\Ng}(c)$ by
\[
\cl_{\Gg,\Ng}(c)=\inf \cl_{\Gg,\Ng}(\xi_1\cdots \xi_m\eta_1^{-1}\cdots \eta_n^{-1}).
\]
Here, $\xi_1,\cdots,\xi_m$ run over the conjugacy classes of $\xl_1,\ldots,\xl_m$ in $\Gg$ and $\eta_1,\cdots,\eta_n$ run over the conjugacy classes of $\check{\xl}_1,\ldots,\check{\xl}_n$ in $\Gg$, respectively.
  \item Let $c\in \CZ(\Gg,\Ng)$ of the form \eqref{eq=Zchain}. We define the \emph{stable mixed commutator length} $\scl_{\Gg,\Ng}(c)$ by
\[
\scl_{\Gg,\Ng}(c)=\lim_{k\to\infty}\frac{\cl_{\Gg,\Ng}(\xl_1^k+\cdots +\xl_m^k- \check{\xl}_1^k-\cdots -\check{\xl}_n^k)}{k}.
\]
  \item Let $c\in \CQ(\Gg,\Ng)$. We define the \emph{stable mixed commutator length} $\scl_{\Gg,\Ng}(c)$ by
\[
\scl_{\Gg,\Ng}(c)=\frac{\scl_{\Gg,\Ng}(kc)}{k},
\]
where we take $k\in \NN$ such that $kc\in \CZ(\Gg,\Ng)$. 
We define $\scl_{\Gg}$ to be $\scl_{\Gg,\Gg}$ on $\mathcal{C}_{\QQ}(\Gg,\Gg)$.
\end{enumerate}
\end{definition}

The inclusion map $i\colon \Ng\hookrightarrow \Gg$ induces a set-theoretic inclusion $\CCC_1(\Ng;R)\hookrightarrow \CCC_1(\Gg;R)$. This map  further induces a map $\iota_{\QQ}\colon \CQ(\Gg,\Ng)\to \CQ(\Gg)$; $c\mapsto c$, which is a homomorphism between additive groups such that for every $c\in \CQ(\Gg,\Ng)$, 
\[
\scl_{\Gg}(\iota_{\QQ}(c))(=\scl_{\Gg}(c))\leq \scl_{\Gg,\Ng}(c).
\]
Hence, this map $\iota_{\QQ}$ can be regarded as a coarse homomorphism 
\[
 \iota_{\QQ}\colon (\CQ(\Gg,\Ng), \sclGN)\to (\CQ(\Gg), \sclG),
\]
and we may discuss the coarse kernel of it; we will discuss more details in Subsection~\ref{subsec=coarse}.

\subsection{Main result for rational chains}\label{subsec=resultsCK}
Here we present  our second main result (Theorem~\ref{mthm=main_Qcoarsekernel}), which is related to the coarse kernel (Definition~\ref{defn=coarsekernel}) in the setting of $\QQ$-chains.
As is discussed in Subsection~\ref{subsec=W}, $i\colon \Ng\hookrightarrow \Gg$ induces an $\RR$-linear map
\[
i^{\ast}\colon \QQQ(\Gg)\to \QQQ(\Ng)^{\Gg};\quad \phf\mapsto i^{\ast}\phf(=\phf\circ i=\phf|_{\Ng}).
\]
This map $i^{\ast}$ induces an $\RR$-linear map
\[
\widehat{i^{\ast}}\colon \QQQ(\Gg)/\HHH^1(\Gg)\to \QQQ(\Ng)^{\Gg}/\HHH^1(\Ng)^{\Gg}.
\]
\relax The quotient $\RR$-linear space $\WGN$ can be regarded as the cokernel of this map:
\[
\WGN=\mathrm{Coker}\big(\widehat{i^{\ast}}\colon \QQQ(\Gg)/\HHH^1(\Gg)\to \QQQ(\Ng)^{\Gg}/\HHH^1(\Ng)^{\Gg}\big).
\]
Recall from the introduction that the two $\RR$-linear spaces $\QQQ(\Gg)/\HHH^1(\Gg)$ and $\QQQ(\Ng)^{\Gg}/\HHH^1(\Ng)^{\Gg}$ are equipped with (genuine) defect norms, respectively. We write $\DD_{\Gg}$ for the defect norm on $\QQQ(\Gg)/\HHH^1(\Gg)$ and $\DD_{\Ng}$ for that on $\QQQ(\Ng)^{\Gg}/\HHH^1(\Ng)^{\Gg}$ to distinguish them.  Analytically, $\widehat{i^{\ast}}$ is a bounded linear operator from  the real normed space  $(\QQQ(\Gg)/\HHH^1(\Gg),\DD_{\Gg})$ to $(\QQQ(\Ng)^{\Gg}/\HHH^1(\Ng)^{\Gg},\DD_{\Ng})$. Indeed, for every $\phf\in \QNG$, we have $\DD_{\Ng}(\widehat{i^{\ast}}[\phf])\leq \DD_{\Gg}([\phf])$; in other words, the operator norm of $\widehat{i^{\ast}}$ is at most $1$. We may regard $\WGN$ as a quotient  seminormed  space: $\WGN$ is a normed space if and only if $\widehat{i^{\ast}}(\QQQ(\Gg)/\HHH^1(\Gg))$ is norm-closed in $(\QQQ(\Ng)^{\Gg}/\HHH^1(\Ng)^{\Gg},\DD_{\Ng})$.

Theorem~\ref{mthm=main_Qcoarsekernel} below describes the relation between a coarse kernel of $\iota_{\QQ}$ and the space $\WGN$.
 Here, we note that a coarse kernel of a coarse homomorphism is naturally equipped with a coarse group structure, as we will discuss in Remark~\ref{rem=coarsekernel}.

\begin{mthm}[coarse kernel for rational chains]\label{mthm=main_Qcoarsekernel}
Let $\Gg$ be a group and $\Ng$ its normal subgroup. Let $i\colon \Ng\hookrightarrow \Gg$ be the inclusion map.  Define
\[
\iota_{\QQ}\colon (\CQ(\Gg,\Ng),\scl_{\Gg,\Ng})\to (\CQ(\Gg),\scl_{\Gg});\quad c\mapsto c.
\]
Then, the following hold.
\begin{enumerate}[label=\textup{(\arabic*)}]
  \item A coarse kernel of the map $\iota_{\QQ}$ exists if and only if $\widehat{i^{\ast}}(\QQQ(\Gg)/\HHH^1(\Gg))$ is norm-closed in $(\QQQ(\Ng)^{\Gg}/\HHH^1(\Ng)^{\Gg},\DD_{\Ng})$.
  \item Assume that $\widehat{i^{\ast}}(\QQQ(\Gg)/\HHH^1(\Gg))$ is norm-closed in $(\QQQ(\Ng)^{\Gg}/\HHH^1(\Ng)^{\Gg},\DD_{\Ng})$. Then, a coarse kernel of $\iota_{\QQ}$ is isomorphic, as a coarse group, to a real normed space $V$ such that the continuous dual $V^{\ast}$ is isometrically isomorphic to $\WGN$, equipped with the quotient norm  of $\DD_{\Ng}$. 
\end{enumerate}
In particular, if $\WGN$ is finite-dimensional, then a coarse kernel of $\iota_{\QQ}$ exists and it is isomorphic to the continuous dual $\WGN^{\ast}$ as a coarse group.
\end{mthm}

As we have seen in Remark~\ref{rem=dimW}, a mild condition on $(\Gg,\Ng)$ ensures the finite-dimensionality of  $\WGN$; we refer the reader to \cite[Theorem 7.11]{KKMMM_KIAS} for a more general condition.
We also have the following result, which relates coarse kernels to the extendability, up to invariant homomorphisms, of invariant quasimorphisms.

\begin{theorem}[coarse kernel and extendability up to invariant homomorphisms]\label{thm=cCKextend}
Let $\Gg$ be a group and $\Ng$ its normal subgroup. Assume that a coarse kernel $A$ of $\iota_{\QQ}\colon (\mathcal{C}_{\QQ}(\Gg,\Ng),\sclGN)\to (\mathcal{C}_{\QQ}(\Gg),\sclG)$ exists. Then, for every $\muf\in \QNG$, the following are equivalent.
\begin{enumerate}[label=\textup{(\roman*)}]
  \item $[\muf]=0$ in $\WGN$. In other words, there exists $\phf\in \QQQ(\Gg)$ such that $\muf-\phf|_{\Ng}\in \HHH^1(\Ng)^{\Gg}$.
  \item $\muf(A)$ is a bounded subset of $\RR$.
\end{enumerate}
\end{theorem}

As  an application  of our theory, we have a crushing theorem, which we will discuss in Subsection~\ref{subsec=crushing} as Theorem~\ref{thm=crushing} and Remark~\ref{rem=crushing}.

\subsection{A criterion for the norm-closedness of the range of $\widehat{i^{\ast}}$}\label{subsec=closedrange}
In relation to Theorem~\ref{mthm=main_Qcoarsekernel}, it is natural to ask when $\widehat{i^{\ast}}(\QQQ(\Gg)/\HHH^1(\Gg))$ is norm-closed in $(\QQQ(\Ng)^{\Gg}/\HHH^1(\Ng)^{\Gg},\DD_{\Ng})$ in the setting of Theorem~\ref{mthm=main_Qcoarsekernel}. One useful sufficient condition is that $\WW(\Gg,\Ng)=\mathrm{Coker}(\widehat{i^{\ast}})$ is finite-dimensional, as we will see in Proposition~\ref{prop=findim}. Furthermore, we obtain another criterion in terms of ordinary and bounded cohomology of groups (we will recall basic facts in bounded cohomology in Subsection~\ref{subsec=H_b}) as follows.

\begin{theorem}\label{thm=closedrange}
Let $\Gg$ be a group and $\Ng$ its normal subgroup. Let $i\colon \Ng\hookrightarrow \Gg$ be the inclusion map. Let $\Gam=\Gg/\Ng$. 
Assume that the following two conditions hold. 
\begin{enumerate}[label=\textup{(\alph*)}]
  \item $\HHH^3_b(\Gam)$ is finite-dimensional.
  \item Either $\HHH^2_b(\Gam)$ or $\HHH^2(\Gam)$ is finite-dimensional.
\end{enumerate}
Then, $\widehat{i^{\ast}}(\QQQ(\Gg)/\HHH^1(\Gg))$ is norm-closed in $(\QQQ(\Ng)^{\Gg}/\HHH^1(\Ng)^{\Gg},\DD_{\Ng})$. 
\end{theorem}

We note that conditions (a) and (b) are satisfied if $\Gam$ is amenable. Indeed, by \cite{Gromov1983}, for every $n\geq 1$ we have $\HHH^n_b(\Gam)=0$ in this case.

\section{Preliminaries}\label{sec=prelim}
\subsection{Quasimorphisms}\label{subsec=qm}
Let $\Gg$ be a group and $\Ng$ its normal subgroup. Recall from Subsection~\ref{subsec=W} that $\QQQ(\Ng)$ denotes the $\RR$-linear space of homogeneous quasimorphisms on $\Ng$, and that $\QQQ(\Ng)^{\Gg}$ means its invariant part under the adjoint action of $\Gg$. If $\phf$ is a quasimorphism on a group $\Gg$, then for all $\gl_1,\gl_2 \in \Gg$ we have 
\[
\phf(\gl_1\gl_2)\sim_{\DD(\phf)} \phf(\gl_1)+\phf(\gl_2);
\]
recall from our notation that $a\sim_C b$ means $|a-b|\leq C$.
We employ the following lemma in the present paper.

\begin{lemma}[{see \cite[Lemma~4.8]{KKMMM_KIAS}}]\label{lem=quasiinv}
Let $\muf\in \QQQ(\Ng)$. Assume that there exists $\DD\in \RR_{\geq 0}$ such that for every $\gl\in \Gg$ and $\xl\in \Ng$, $\muf(\gl\xl\gl^{-1})\sim_{\DD}\muf(\xl)$.
Then, $\muf\in \QQQ(\Ng)^{\Gg}$. In particular, we have $\QQQ(\Gg)=\QQQ(\Gg)^{\Gg}$.
\end{lemma}

\subsection{Bavard duality and its strengthenings}\label{subsec=Bavard}

Let $\Gg$ be a group. The \emph{Bavard duality theorem} for $\Gg$, established by Bavard \cite{Bavard},  states that for every $\hl\in [\Gg,\Gg]$,
\begin{equation}\label{eq=Bavard}
\scl_{\Gg}(\hl)=\sup_{[\phf]\in (\QQQ(\Gg)/\HHH^1(\Gg))\setminus \{0\}}\frac{|\phf(\hl)|}{2\DD_{\Gg}(\phf)}.
\end{equation}
Three strengthenings of the Bavard duality have been obtained as in Figure~\ref{fig=visual}, as we will give an overview in this subsection. Here, we set $\BBB_1(\Gg;\RR)=\partial \CCC_2(\Gg;\RR)$ in the setting of Subsection~\ref{subsec=Qchains}, which equals $\BBB'_1(\Gg,\Gg;\RR)$.

\begin{figure}[h]
  \centering
    \begin{tikzpicture}[auto]
    \node[draw, align=center] (01) at (0, 2) {original (\cite{Bavard}): \\ for $\hl \in \CG$};
    \node[draw, align=center] (11) at (7, 2) {generalized  (\cite{Calegari}): \\ for $c \in \BBB_1(\Gg;\mathbb{R})$};
    \node[draw, align=center] (00) at (0, 0) {mixed  (\cite{KKMM1}): \\ for $\yl \in \CGN$};
    \node[draw, align=center] (10) at (7, 0) {generalized mixed (\cite{KKMMM_KIAS}): \\ for $c \in \CR(\Gg,\Ng)$ };
    \draw[<-] (01) to node {$\Ng=\Gg$} (00);
    \draw[<-] (01) to node {$c=\hl$} (11);
    \draw[<-] (11) to node {$\Ng=\Gg$} (10);
    \draw[<-] (00) to node {$c=\yl$} (10);
    \end{tikzpicture}
    \caption{The original Bavard duality and its three strengthenings}  \label{fig=visual}
\end{figure}
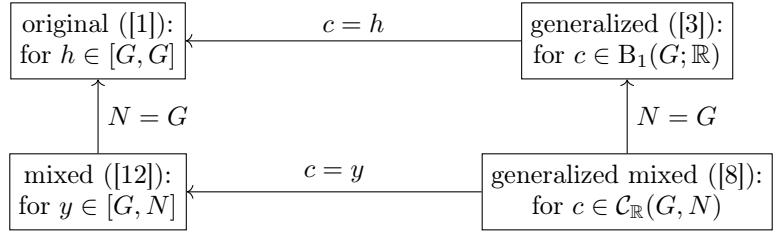

First, we state the generalized mixed Bavard duality, obtained in \cite{KKMMM_KIAS}. Relevant definitions are given in Subsection~\ref{subsec=Qchains}. Here, note that the value of $\muf(c)$ does not depend on the representative $\muf$ of $[\muf]$ because $c\in \mathcal{C}_{\QQ}(\Gg,\Ng)$  (see Lemma \ref{lem=CGNchain} below).

\begin{theorem}[{generalized mixed Bavard duality theorem for rational chains, \cite[Theorem~5.39]{KKMMM_KIAS}}]\label{thm=gmBavard}
Let $\Gg$ be a group and $\Ng$ its normal subgroup. Then, for every $c\in \mathcal{C}_{\QQ}(\Gg,\Ng)$, we have
\[
\scl_{\Gg,\Ng}(c)=\sup_{[\muf]\in (\QQQ(\Ng)^{\Gg}/\HHH^1(\Ng)^{\Gg})\setminus \{0\}}\frac{|\muf(c)|}{2\DD_{\Ng}(\muf)}.
\]
\end{theorem}

By letting $\Ng=\Gg$ in Theorem~\ref{thm=gmBavard}, we recover the \emph{generalized Bavard duality theorem} (for rational chains), which was proved by Calegari \cite{Calegari}.  By letting $c=\yl$ for $\yl\in \CGN$ in Theorem~\ref{thm=gmBavard}, we recover the \emph{mixed Bavard duality theorem} \eqref{eq=mixedBavard}: this  was obtained in \cite{KKMM1} in the full generality. For this reduction, we use the following lemma.

\begin{lemma}[{see \cite[Lemma~5.28]{KKMMM_KIAS}}]\label{lem=CGNchain}
Let $\Gg$ be a group and $\Ng$ its normal subgroup. Then the following are equivalent for $c \in \CCC_1(\Ng; \ZZ)$.
\begin{enumerate}[label=\textup{(\arabic*)}]
\item $c$ belongs to $\CZ(\Gg,\Ng)$.
\item There exist $m,n\in \ZZ_{\geq 0}$ and $\xl_1, \cdots, \xl_m, \xbr_1, \cdots, \xbr_n \in \Ng$ such that
\[ c = \xl_1 + \cdots + \xl_m - \xbr_1 - \cdots - \xbr_n\]
and $\xl_1 \cdots \xl_m \xbr_1^{-1} \cdots \xbr_n^{-1} \in [\Gg,\Ng]$.
\end{enumerate}
In particular, for every $\yl\in \CGN$, the element $\yl$ in $\CCC_1(\Ng;\ZZ)$ belongs to $\mathcal{C}_{\ZZ}(\Gg,\Ng)$.
\end{lemma}
Finally, by letting $c=\hl$ in the generalized Bavard duality theorem, or by letting $\Gg=\Ng$ in \eqref{eq=mixedBavard}, we recover the original Bavard duality theorem \eqref{eq=Bavard} (by Lemma~\ref{lem=quasiinv}, $\QQQ(\Gg)/\HHH^1(\Gg)=\QQQ(\Gg)^{\Gg}/\HHH^1(\Gg)^{\Gg}$).

As we discussed in Subsection~\ref{subsec=strategy}, a \emph{refined version} of the generalized mixed Bavard duality theorem (Theorem~\ref{thm=gmBavard}) plays a central role in the present paper; we present this as Theorem~\ref{thm=refinedgmBavard}.

\subsection{Basic facts on functional analysis}\label{subsec=FA}

Here we recall standard results on functional analysis; we refer the reader to \cite{Rudin} for a treatise on functional analysis. For a real normed space $(X,\|\cdot\|_X)$, we write $\langle \cdot,\cdot\rangle \colon X^{\ast}\times X\to \RR$ for the duality pairing \relax for $X$. Here, the continuous dual $X^{\ast}$ is equipped with the dual norm
\[
\|\cdot\|_{X^{\ast}}\colon X^{\ast}\to \RR_{\geq 0};\quad \|\phi\|_{X^{\ast}}=\sup_{x\in X\setminus\{0\}}\frac{|\langle \phi,x\rangle|}{\|x\|_X}.
\]
\relax Strictly speaking, for two real normed spaces $(X,\|\cdot\|_X)$ and $(Y,\|\cdot\|_Y)$, we call an $\RR$-bilinear form $\langle \cdot,\cdot\rangle \colon Y\times X\to \RR$ the \emph{duality pairing for $X$} if the map 
\[
Y\to \mathrm{Hom}_{\RR}(X,\RR);\ y\mapsto \langle y,\cdot \rangle, \quad \textrm{where}\quad \langle y,\cdot \rangle\colon X\to \RR;\ x\mapsto \langle y,x\rangle,
\]
induces a surjective map from $Y$ to $X^{\ast}$ that is isometric (namely, for every $y\in Y$, we have $\|\langle y,\cdot\rangle\|_{X^{\ast}}=\|y\|_Y$). Here, we naturally regard $X^{\ast}$ as a subset of  $\mathrm{Hom}_{\RR}(X,\RR)$. For a subspace $U\subseteq X$ and $V\subseteq X^{\ast}$, two norm-closed subspaces $U^{\perp}\subseteq X^{\ast}$ and ${}^{\perp}V\subseteq X$ are defined as
\begin{align*}
U^{\perp}&=\{\varphi\in X^{\ast}\;|\; \langle \varphi,x\rangle=0\mathrm{\ for\ every\ }x\in U\},\\
{}^{\perp}V&=\{x\in X\;|\; \langle \varphi,x\rangle=0\mathrm{\ for\ every\ }\varphi\in V\},
\end{align*}
respectively. For a subspace $U$ of a real normed space $(X,\|\cdot\|_X)$, we equip the quotient space $X/U$ with the quotient seminorm
\[
\|\cdot\|_{X/U}\colon X/U\to \RR_{\geq 0};\quad \|[x]\|_{X/U}=\inf_{u\in U}\|x+u\|_{X}.
\]
We note that $\|\cdot\|_{X/U}$ is a genuine norm if and only if  $U$ is a norm-closed subspace in $X$. Hence, when taking the quotient space $X/U$ in analytic setting, we mainly focus on the case where $U$ is norm-closed in $X$. A linear operator $S\colon (X,\|\cdot\|_X)\to (Y,\|\cdot\|_Y)$ between two real normed spaces is called a bounded (linear) operator if its operator norm
\[
\|S\|_{\op}=\sup_{x\in X\setminus \{0\}}\frac{\|Sx\|_Y}{\|x\|_X}
\]
is finite. For a bounded linear operator $S\colon X\to Y$ between two real normed spaces, $S^{\dagger}\colon Y^{\ast}\to X^{\ast}$ denotes the adjoint operator, that is, the operator defined by
\[
\langle S^{\dagger}\psi,x\rangle_X=\langle \psi,Sx\rangle_Y
\]
for every $x\in X$ and $\psi\in Y^{\ast}$. Here, we use the symbol $S^{\dagger}$, instead of $S^{\ast}$, to distinguish it from $i^{\ast}\colon \QQQ(\Gg)\to \QQQ(\Ng)^{\Gg}$ (and other homomorphisms appearing in group cohomology). Then, we have $\|S^{\dagger}\|_{\op}=\|S\|_{\op}$ (compare with Proposition~\ref{prop=FA}~(1) below).

\begin{lemma}[{see \cite[Theorem~4.9]{Rudin}}]\label{lem=dualquotient}
Let $X$ be a real normed space and $U$ a subspace. 
\begin{enumerate}[label=\textup{(\arabic*)}]
  \item The duality pairing $\langle \cdot,\cdot\rangle_X \colon X^{\ast}\times X\to \RR$ induces a well-defined pairing
\[
\langle \cdot,\cdot\rangle_U\colon \big(X^{\ast}/U^{\perp}\big)\times U  \to \RR;  \quad \langle [\phi],x\rangle_U =\langle \phi,x\rangle_X,
\]
and this $\langle \cdot,\cdot\rangle_U$ is the duality pairing for $U$. In particular, $U^{\ast}$ is isometrically isomorphic to $X^{\ast}/U^{\perp}$.
  \item Assume that $U$ is norm-closed in $X$. Then the duality pairing $\langle \cdot,\cdot\rangle_X$ induces a well-defined pairing
\[
\langle \cdot,\cdot\rangle_{X/U}\colon U^{\perp}\times X/U \to \RR;   \quad \langle \phi,[x]\rangle_{X/U} =\langle \phi,x\rangle_X,
\]
and this $\langle \cdot,\cdot\rangle_{X/U}$ is the duality pairing for $X/U$. In particular, $(X/U)^{\ast}$ is isometrically isomorphic to $U^{\perp}$.
\end{enumerate}
\end{lemma}

\begin{theorem}[{Banach's closed range theorem for bounded operators; see \cite[Theorems~4.14, 4.12 and 4.7]{Rudin}}]\label{thm=Banach}
Let $X,Y$ be two real Banach spaces and $S\colon X\to Y$ a bounded linear operator. Then, the following are all equivalent.
\begin{enumerate}[label=\textup{(\roman*)}]
  \item $S(X)$ is norm-closed in $Y$.
  \item $S^{\dagger}(Y^{\ast})$ is norm-closed in $X^{\ast}$.
  \item $S(X)={}^{\perp}\Ker(S^{\dagger})$.
  \item $S^{\dagger}(Y^{\ast})=\Ker(S)^{\perp}$.
\end{enumerate}
\end{theorem}

\begin{remark}\label{rem=complete}
We warn that Banach's closed range theorem will \emph{fail} if we drop the condition of completeness of $X$ or $Y$. For instance, assume that $S$ satisfies one of the four (equivalent) conditions in Theorem~\ref{thm=Banach}.  Take a proper dense subspace $X_0\subseteq X$ and set $S_0=S|_{X_0}$. Then, $X_0^{\ast}$ equals $X^{\ast}$ and $S_0^{\dagger}$ equals $S^{\dagger}$. Hence, the counterpart of (ii) for $S_0$ holds; namely, $S_0^{\dagger}(Y^{\ast})$ is norm-closed in $X_0^{\ast}$. However, there is \emph{no} guarantee that the counterpart of (iv), $S_0^{\dagger}(Y^{\ast})=\Ker(S_0)^{\perp}$, holds. Indeed, we have $S_0^{\dagger}(Y^{\ast})=S^{\dagger}(Y^{\ast})$ equals $\Ker(S)^{\perp}$ by (iv), but $\Ker(S_0)=\Ker(S)\cap X_0$ 
can be considerably small compared with $\Ker(S)$ itself. 
\end{remark}

We also recall the following results, which are corollaries to the Hahn--Banach theorem and the open mapping theorem, respectively. 

\begin{proposition}[{see \cite[Theorem~4.3 and Corollary~2.12]{Rudin}}]\label{prop=FA}
The following hold true.
\begin{enumerate}[label=\textup{(\arabic*)}]
  \item Let $(X,\|\cdot\|_X)$ be a real normed space and $\langle \cdot,\cdot \rangle \colon (X^{\ast},\|\cdot\|_{X^{\ast}})\times (X,\|\cdot\|_X)\to\RR$ the duality pairing. Then, for every $x\in X$ we have
\[
\|x\|_X=\sup_{\phi\in X^{\ast}\setminus \{0\}}\frac{|\langle \phi,x\rangle |}{\|\phi\|_{X^{\ast}}}.
\]
 \item Let $X$ and $Y$ be two real Banach spaces. Let $S\colon X\to Y$ be a bounded linear operator. Assume that $S$ is bijective. Then, $S$ is an isomorphism between Banach spaces, namely, the set-theoretic inverse map $S^{-1}\colon Y\to X$ is a bounded operator.
\end{enumerate}
\end{proposition}

\relax The following standard fact is a  corollary to Proposition~\ref{prop=FA}~(2); for the reader's convenience, we include the proof.

\begin{corollary}\label{cor=findim}
Let $X$ and $Y$ be two real Banach spaces. Let $S\colon X\to Y$ be a bounded linear operator. Assume that the quotient $\RR$-linear space $Y/S(X)$ in algebraic sense is finite-dimensional. Then, $S(X)$ is norm-closed in $Y$.
\end{corollary}

\begin{proof}
By replacing $S\colon X\to Y$ by the induced operator $X/\Ker(S)\to Y$, we may assume that $S$ is injective. Let $\ell=\Rdim (Y/S(X))$. Take a complete system $\{y_1,\ldots,y_{\ell}\}$ of representatives for $Y\twoheadrightarrow Y/S(X)$, and define the following linear operator 
\[
T\colon X\oplus_1 (\RR^{\ell},\|\cdot\|_1)\to Y;\quad (x,(t_1,\ldots ,t_{\ell}))\mapsto Sx+t_1y_1+\cdots+t_{\ell}y_{\ell}.
\]
Here $\oplus_1$ denotes the $\ell^1$-direct sum and $\|\cdot\|_1$ means the $\ell^1$-norm on $\RR^{\ell}$. Then, $T$ is bijective and $\|T\|_{\op}\leq \|S\|_{\op}+\max\limits_{1\leq i\leq \ell}\|y_i\|$. By Proposition~\ref{prop=FA}~(2), $T$ is an isomorphism. In particular, $S(X)=T(X\oplus \{(0,\ldots,0)\})$ is norm-closed in $Y$.
\end{proof}

\subsection{Coarse groups and coarse kernels}\label{subsec=coarse}

Here we briefly discuss the theory of coarse groups, which is developed by Leitner and Vigolo \cite{LV}. Coarse groups are defined to be group objects in the category of coarse spaces. However, in this Part I paper, we focus on the case where a coarse group is given by a set-group $\Gg$ with a bi-invariant metric $d_{\Gg}$; the study of general coarse groups will be needed in our Part II paper \cite{KKMMM_partII}. The main topic in this subsection is the \emph{coarse kernel} of a coarse homomorphism.

\begin{definition}
Let $(X,d_{X})$ and $(Y,d_{Y})$ be two metric spaces. Let $f,f'\colon X \to Y$ be maps.
\begin{enumerate}[label=(\arabic*)]
  \item The map $f$ is said to be \emph{bornologous} if for every $R\in \mathbb{R}_{\geq 0}$, there exists $\hat{R}\in \mathbb{R}_{\geq 0}$ such that for all $x_1,x_2\in X$, $d_X(x_1,x_2)\leq R$ implies that $d_Y(f(x_1),f(x_2))\leq \hat{R}$.
  \item We say that $f$ and $f'$ are \emph{close}, written as $f\sim f'$, if there exists $C\in \mathbb{R}_{\geq 0}$ such that for every $x\in X$, $d_X(f(x),f'(x))\leq C$ holds. 
\end{enumerate}
\end{definition}
In the coarse category, a morphism, called a \emph{coarse map}, between two metric spaces is the equivalence class of bornologous maps with respect to $\sim$. Hence, strictly speaking, a coarse map is not a set-map. However, by abuse of notation, we use terminologies for a coarse map to a set-map that represents this coarse map. The concept of a coarse homomorphism is one of such examples. 

Let $(X,d_{X})$ be a metric space. For every $R\in \mathbb{R}_{\geq 0}$, we define a map on the power set $N^{(X,d_X)}_R\colon \mathcal{P}(X)\to \mathcal{P}(X)$
by letting for every $A\in \mathcal{P}(X)$, $N^{(X,d_X)}_R(A)$ to be the $R$-neighborhood of $A$. Namely, we set
\[
N^{(X,d_X)}_R(A)=\{x\in X\;|\; \mathrm{dist}(x,A)\leq R\}.
\]
Here, $\mathrm{dist}$ denotes the distance: for $A,B\subseteq X$, 
\[
\mathrm{dist}(A,B)=\inf\{d_X(a,b)\;|\;a\in A,\,b\in B\},
\]
and $\mathrm{dist}(x,A)$ is defined to be $\mathrm{dist}(\{x\},A)$. Then the closed ball $B^{(X,d_X)}_R(x)$ (recall our notation) equals $N^{(X,d_X)}_R(\{x\})$. By using this map $N^{(X,d_X)}_R$, we can define the asymptotic equivalence relation on $\mathcal{P}(X)$ as follows.

\begin{definition}\label{defn=asymptotic}
Let $(X,d_{X})$ be a metric space, and $A,B\subseteq X$.
\begin{enumerate}[label=(\arabic*)]
  \item We write $A\precsim B$ if there exists $R\in \RR_{\geq 0}$ such that $A\subseteq N^{(X,d_X)}_{R}(B)$. In other words, there exists $R'\in \RR_{\geq 0}$ such that for every $a\in A$, there exists $b\in B$ that satisfies $d_X(a,b)\leq R'$. We also write $B\succsim A$ if $A\precsim B$.
  \item We say that $A$ and $B$ are \emph{asymptotic}, written as $A\asymp B$, if $A\precsim B$ and $A\succsim B$ both hold.
\end{enumerate}
\end{definition}

A \emph{coarse subset} of $X$ is an equivalence class with respect to the asymptotic relation $\asymp$. 
It is straightforward to see that for bornologous maps $f,f'\colon (X,d_X)\to (Y,d_Y)$ and for $A\subseteq X$, if $f\sim f'$, then $f(A)\asymp f'(A)$. Therefore, for a coarse map $\mathbf{f}=[f]=[f]_{\sim}$ and a coarse set $\mathbf{A}=[A]=[A]_{\asymp}$, the coarse image $\mathbf{f}(\mathbf{A})=[f(A)]=[f(A)]_{\asymp}$ is well-defined. Now, we are in position to define a coarse homomorphism between two groups equipped with bi-invariant metrics and a coarse kernel of it.

\begin{definition}[coarse homomorphism]\label{defn=coarsehom}
Let $(\Gg,\cdot_{\Gg},d_{\Gg})$ and $(\Hg,\cdot_{\Hg},d_{\Hg})$ be two groups equipped with bi-invariant metrics.
\begin{enumerate}[label=(\arabic*)]
  \item We say that $f\colon (\Gg,d_{\Gg})\to (\Hg,d_{\Hg})$ is a \emph{coarse homomorphism} if the following two conditions are satisfied.
  \begin{enumerate}[label=(\alph*)]
    \item $f$ is bornologous.
    \item There exists $D\in \RR_{\geq 0}$ such that for all $\gl_1,\gl_2\in \Gg$, we have
\[
d_{\Hg}(f(\gl_1\cdot_{\Gg}\gl_2),f(\gl_1)\cdot_{\Hg}f(\gl_2))\leq D.
\]
  \end{enumerate}
  \item We say that $(\Gg,d_{\Gg})$ and $(\Hg,d_{\Hg})$ are \emph{isomorphic} as coarse groups if there exist two coarse homomorphisms $f\colon (\Gg,d_{\Gg})\to (\Hg,d_{\Hg})$ and $f'\colon (\Hg,d_{\Hg})\to (\Gg,d_{\Gg})$ such that $f'\circ f\sim \mathrm{id}_{\Gg}$ and $f\circ f'\sim \mathrm{id}_{\Hg}$.
\end{enumerate}
\end{definition}

\begin{definition}[coarse kernel]\label{defn=coarsekernel}
Let $(\Gg,d_{\Gg})$ and $(\Hg,d_{\Hg})$ be two groups equipped with bi-invariant metrics. Let $f\colon (\Gg,d_{\Gg})\to (\Hg,d_{\Hg})$ be a coarse homomorphism. Then, we say that $A\subseteq \Gg$ is a \emph{coarse kernel} of $f$ if the following two conditions are both satisfied.
  \begin{enumerate}[label=(\alph*)]
    \item The set $f(A)$ is bounded in $d_{\Hg}$. In other words, there exists $R\in \RR_{\geq 0}$ such that $f(A)\subseteq N^{(\Hg,d_{\Hg})}_R(\{e_{\Hg}\})  (=B^{(\Hg,d_{\Hg})}_R(e_{\Hg})). $
    \item The set $A$ is maximal with respect to $\precsim$ among all $B\subseteq  \Gg$ such that $f(B)$ is bounded in $d_{\Hg}$. That is, for every $B\subseteq  \Gg$ such that $f(B)$ is bounded in $d_{\Hg}$, we have $B\precsim A$.
  \end{enumerate}
\end{definition}

The kernel of a genuine group homomorphism $f\colon \Gg\to \Hg$ is the maximal subset $A$ of $\Gg$, with respect to $\subseteq$, that satisfies $f(A)=\{e_{\Hg}\}$; hence, Definition~\ref{defn=coarsekernel} `coarsifies' the concept of the kernel. Strictly speaking, a coarse homomorphism should be defined for a coarse map $\mathbf{f}=[f]$ and that of `the' coarse kernel should be defined to be a coarse subset $\mathbf{A}=[A]$. Whereas a set-theoretic coarse kernel may not be unique in general, a coarse kernel $\mathbf{A}=[A]$, as an asymptotic class, of a coarse homomorphism $f$ is uniquely determined by definition, provided that it exists. In addition, the coarse kernel $\mathbf{A}$ does not depend on a representative $f$ of $\mathbf{f}=[f]$. 

We warn that a coarse kernel of a coarse homomorphism may \emph{not} exist  in general; compare with Proposition~\ref{prop=CKBanach} below. If a coarse kernel exists, then we have the coarse isomorphism theorem; see \cite[Chapter~7]{LV} for details.

\begin{example}\label{exa=coarsegroup}
Let $(X,\|\cdot\|_X)$ be a real normed  space. View $X$  as an additive group $(X,+)$. We note $\|\cdot\|_X$  determines a bi-invariant metric on $X$. Thus, we can regard $(X,\|\cdot\|_X)$ as a coarse group. Let $(Y,\|\cdot\|_Y)$ be also a real normed  space, and $T\colon (X,\|\cdot\|_X)\to (Y,\|\cdot\|_Y)$  be a bounded linear operator. Then $T$ is in particular an additive group homomorphism, and $T$ is bornologous because $\|T\|_{\mathrm{op}}<\infty$. Hence, $T$ may be regarded as a coarse homomorphism.
\end{example}

\begin{remark}\label{rem=coarsekernel}
 We note that a coarse kernel in Definition \ref{defn=coarsekernel} may not be closed under the original group multiplication.  Here we explain that nevertheless, we can equip this with a structure of a coarse group. 
Let $f\colon (\Gg,d_{\Gg})\to (\Hg,d_{\Hg})$ be a coarse homomorphism between two groups with bi-invariant metrics. Assume that a coarse kernel $A$ of $f$ exists. By replacing $A$ with $A\cup\{e_{\Gg}\}$, we may assume that $e_{\Gg}\in A$. Since $f(A)$ is bounded in $(\Hg,d_{\Hg})$, there exists $R\in \RR_{\geq 0}$ such that $f(A)\subseteq B^{(\Hg,d_{\Hg})}_R(e_{\Hg})$. Since $A$ is a coarse kernel, we have $f^{-1}(B^{(\Hg,d_{\Hg})}_{2R}(e_{\Hg}))\precsim A$.
In other words, there exist $D\in \RR_{\geq 0}$ and a map $\Phi\colon f^{-1}(B^{(\Hg,d_{\Hg})}_{2R}(e_{\Hg}))\to A$ such that for every $x\in f^{-1}(B^{(\Hg,d_{\Hg})}_{2R}(e_{\Hg}))$, we have $d_{\Gg}(\Phi(x),x)\leq D$. We can choose this map $\Phi$ to satisfy $\Phi(a)=a$ for every $a\in A$. Thus, we can define the $\star_{\Phi}\colon A\times A\to A$; $a\star_{\Phi}b=\Phi(a\cdot_{\Gg}b)$. By construction, for all $a,b\in A$ we have
\begin{equation}\label{eq=coarsehomA}
d_{\Gg}(a\star_{\Phi}b,a\cdot_{\Gg}b)\leq D.
\end{equation}
Additionally, $a\star b=a\cdot_{\Gg}b$ if $a\cdot_{\Gg}b\in A$.
Thus, by  equipping $A$ with the `coarse group multiplication' $\star_{\Phi}$, we may regard $(A,e_{\Gg},\star_{\Phi},d_{\Gg})$ as a coarse group. Then, we may view this as a coarse subgroup of $(\Gg,e_{\Gg},\cdot_{\Gg},d_{\Gg})$. Here, we do not explain the precise meaning of these assertions; we refer the reader to \cite{LV} instead. At the levels of $\mathbf{f}=[f]$ and $\mathbf{A}=[A]$, the coarse group structure (up to coarse isomorphisms) of $\mathbf{A}$ is independent of the choices of representatives $f$ and  $A$, or of the choice of $\Phi$.
\end{remark}

The following  theorem  treats the case of Example~\ref{exa=coarsegroup} when $X$ and $Y$ are both complete. In the present paper, we only treat  the special case  where coarse groups are set-groups equipped with bi-invariant metrics. Hence, we include  the proof of this  theorem  within our setting for the convenience of the reader.

\begin{theorem}[{\cite[Theorem 12.2.5]{LV}}]\label{prop=CKBanach}
Let $(X,\|\cdot\|_X)$ and $(Y,\|\cdot\|_Y)$ be two real Banach spaces and $S\colon X\to Y$ a bounded linear operator. Then the following hold true.
\begin{enumerate}[label=\textup{(\arabic*)}]
  \item A coarse kernel of $S$ exists if and only if $S(X)$ is norm-closed in $Y$. 
  \item If $S(X)$ is norm-closed in $Y$, then $\Ker (S)$ is a coarse kernel of $S$.
\end{enumerate}
\end{theorem}

\begin{proof}
Let $\|\cdot\|_S$ be the quotient norm on $X/\Ker(S)$. Then, $S$ descends to  an injective bounded  operator $\check{S}\colon (X/\Ker(S),\|\cdot\|_S)\to (Y,\|\cdot\|_Y)$. Set $T\colon (S(X),\|\cdot\|_Y)\to (X/\Ker(S),\|\cdot\|_S)$ as the linear operator uniquely determined by the condition  $T\circ \check{S}=\mathrm{Id}_{X/\Ker(S)}$. Then, by  Proposition~\ref{prop=FA}~(2)  together with a standard argument using a Cauchy sequence, $S(X)$ is norm-closed in $Y$ if and only if $\|T\|_{\mathrm{op}}<\infty$.

First, assume that $S(X)$ is norm-closed in $Y$. Set $C=\|T\|_{\mathrm{op}}$, which is finite. Set $A=\Ker(S)$. Then $S(A)=\{0\}$ is bounded. Let $B\subseteq (X,\|\cdot\|_X)$ be a set whose image $S(B)$ is bounded. Take $R\in \RR_{\geq 0}$ such that $S(B)\subseteq B_{(Y,\|\cdot\|_Y)}(0,R)$. Then, for every $b\in B$ we have
\[
\mathrm{dist}(b,A)=\|[b]\|_S=\|(T\circ \check{S})[b]\|_S\leq C \|\check{S}[b]\|_S\leq CR.
\]
Hence, $B\precsim A$. This proves that $A=\Ker(S)$ is a coarse kernel of $S$.  This completes the proofs of ``if'' part of (1) and (2). 

 To see the ``only if'' part of (1),  assume that a coarse kernel of $S$ exists. We will first show that $\Ker(S)$ is a coarse kernel of $S$. To show this, it suffices to prove that $A\precsim \Ker(S)$ for some (or, equivalently every) coarse kernel $A$ of $S$. Let $A_0$ be a coarse kernel of $S$. Then, there exists $R\in \RR_{\geq 0}$ such that $A_0\subseteq S^{-1}(B^{(Y,\|\cdot\|_Y)}_R(0))$. Set $A=S^{-1}(B^{(Y,\|\cdot\|_Y)}_R(0))$. Then, $A$ is also a coarse kernel of $S$; $A$ is norm-closed, and $ta\in A$ holds whenever $a\in A$ and $t\in [0,1]$. Now, as we have discussed in Remark~\ref{rem=coarsekernel}, $A$ admits a coarse group multiplication $\star\colon A\times A\to A$. More precisely, there exists $D\in \RR_{\geq 0}$ such that for all $a,b\in A$,
\begin{equation}\label{eq=neighbor}
\|a\star b-a-b\|\leq D
\end{equation}
holds; compare with \eqref{eq=coarsehomA}. Fix $a\in A$. For each $n\in \ZZ_{\geq 0}$, we set $a_0=a$ and $a_{n+1}=a_n\star a_n$. Then we have a sequence $(a_n)_{n\in \ZZ_{\geq 0}}$ in $A$. By \eqref{eq=neighbor}, for every $n\in \ZZ_{\geq 0}$ we have $\|a_{n+1}/2^{n+1}-a_n/2^n\|\leq (1/2^{n+1})D$. Hence, $(a_n/2^n)_{n\in \ZZ_{\geq 0}}$ is a Cauchy sequence in $A$. By completeness of $(A,\|\cdot\|_X)$, the (norm-)limit $\lim\limits_{n\to\infty}(a_n/2^n)$ exists in $A$. We write $g(a)$ for the limit, associated to the fixed $a\in A$. Then, by construction, for every $a\in A$ we have $\|a-g(a)\|\leq D$ and $g(a\star a)=2g(a)$. The latter equality above implies that $2^ng(a)\in A$ for every $a\in A$ and every $n\in \ZZ_{\geq 0}$. Since $A=S^{-1}(B^{(Y,\|\cdot\|_Y)}_R(0))$, this implies that $g(A)\subseteq \Ker(S)$. Therefore, we conclude that $A\subseteq N^{(X,\|\cdot\|_X)}_D(\Ker(S))$. Hence $A\precsim \Ker(S)$, as desired. Now, suppose that $S(X)$ is not norm-closed in $Y$. Then, $\|T\|_{\mathrm{op}}=\infty$, and there exists a sequence $(x_n)_{n\in \NN}$ in $X$ such that for every $n\in \NN$,
\[
\mathrm{dist}(x_n,\Ker(S))(=\|[x_n]\|_S)\geq n \quad \mathrm{and}\quad x_n\in S^{-1}(B^{(Y,\|\cdot\|_Y)}_1(0)).
\]
This is a contradiction, because we have argued that $\Ker(S)$ is a coarse kernel of $S$. Therefore, $S(X)$ must be norm-closed in $Y$ if a coarse kernel of $S$ exists. 
\end{proof}

\subsection{Bounded cohomology}\label{subsec=H_b}
Let $\Gg$ be a group. Let $n \in \mathbb{Z}$. For $n \ge 0$, let $\CCC^n(\Gg)$ be the $\RR$-linear space of real-valued functions on the $n$-fold direct product $\Gg^n$ of $\Gg$. For $n < 0$, set $\CCC^n(\Gg) = 0$. Define the coboundary map $\delta \colon \CCC^n(\Gg) \to \CCC^{n+1}(\Gg)$ by
\begin{align*}
& \delta c(g_1, \cdots, g_{n+1}) \\
={} & c(g_2, \cdots, g_{n+1}) + \sum_{i=1}^n (-1)^i c(g_1, \cdots, g_{i} g_{i+1}, \cdots, g_{n+1}) + (-1)^{n+1} c(g_1, \cdots, g_{n})
\end{align*}
for $n > 0$, and set $\delta = 0$ for $n \le 0$. The \emph{$n$-th group cohomology $\HHH^n(\Gg)$ of $\Gg$ with trivial $\RR$-coefficients} is the $n$-th cohomology group of the cochain complex $(\CCC^{\bullet}(\Gg), \delta)$. In particular, $\HHH^1(\Gg)$ can be naturally identified with the $\RR$-linear space $\mathrm{Hom}(\Gg,\RR)$ consisting of group homomorphisms from $\Gg$ to $\RR$.

Let $\CCC^n_b(\Gg)$ be the $\RR$-linear space consisting of bounded functions on $\Gg^n$ for $n \ge 0$, and set $\CCC^n_b(\Gg)=0$ for $n<0$. Then $\CCC^\bullet_b(\Gg)$ is a subcomplex of $\CCC^\bullet(\Gg)$, and define $\HHH^{\bullet}_b(\Gg)$ to be the cohomology of $\CCC^{\bullet}_b(\Gg)$. This cohomology group is called the \emph{bounded cohomology of $\Gg$ with trivial $\RR$-coefficients}. The map $c_G^\bullet \colon \HHH^\bullet_b(\Gg) \to \HHH^\bullet(\Gg)$ induced by the inclusion $\CCC^\bullet_b(\Gg) \hookrightarrow \CCC^\bullet(\Gg)$ is called the \emph{comparison map}.

We will  employ the following four facts in Section~\ref{sec=closedrange}. 
\begin{proposition}[{see \cite[Theorem~2.50]{Calegari}}]\label{prop_Q_exact}
The coboundary map $\delta\colon \CCC^1(\Gg)\to \CCC^2(\Gg)$ induces an isomorphism between  the quotient space $\QQQ(\Gg)/\HHH^1(\Gg)$ and the kernel of $c_{\Gg}=c^2_{\Gg} \colon \HHH^2_b(\Gg)\to \HHH^2(\Gg)$. Namely, the following sequence is exact:
\[ 0 \to \HHH^1(\Gg) \to \QQQ(\Gg) \xrightarrow{\delta} \HHH^2_b(\Gg) \xrightarrow{c_{\Gg}} \HHH^2(\Gg).\]
\end{proposition}

\begin{theorem}[{\cite[Example~12.4.3]{Monodbook}}] \label{theorem_Monod_sequence}
Let
\[ 1 \to \Ng \xrightarrow{i} \Gg \xrightarrow{\ppi} \Gamma \to 1 \]
be a short exact sequence of groups. Then there  exists  an exact sequence
\[ 0 \to \HHH^2_b(\Gamma) \xrightarrow{\ppi^*} \HHH^2_b(\Gg) \xrightarrow{i^*} \HHH^2_b(\Ng)^{\Gg} \to \HHH^3_b(\Gamma).\]
\end{theorem}

We equip $\HHH^n_b(G)$ with the seminorm $\|\cdot\|$ induced by the $\ell^{\infty}$-norm on $\CCC^n_b(G)$.

\begin{theorem}[{see \cite[Corollary~6.7]{Frigerio}}] \label{H2b_Banach}
For a group $\Gg$, $(\HHH^2_b(\Gg),\|\cdot\|)$ is a Banach space.
\end{theorem}

\begin{lemma}[{see \cite[Lemma~2.58]{Calegari}}] \label{lem=eqnorms}
For each $\phi \in \QQQ(\Gg)$, we have
\[ \| [\delta \phi]\| \le \DD_\Gg(\phi) \le 2 \cdot \| [\delta \phi]\|.\]
\end{lemma}

\section{Refined version of the generalized mixed Bavard duality theorem}\label{sec=refinedBavard}
 The goal of this section is to obtain Theorem~\ref{thm=refinedgmBavard}, the refined version of the generalized mixed Bavard duality theorem. 

\subsection{$\Ng$-quasimorphisms and the norm $\|\cdot\|'$}\label{subsec=N-quasimorphisms}
Before proceeding to the discussion of our refined version of the generalized mixed Bavard duality theorem, we review the notions of $\Ng$-quasimorphisms and the norm $\|\cdot\|'$ on the space $\BBB'_1(\Gg,\Ng;\RR)$ of $1$-chains. We refer the reader to \cite{KKMMM_KIAS} for more details, including the proofs that are omitted in the present paper. We note that   $\Ng$-quasimorphisms can be seen  as a special case of partial quasimorphisms  (see \cite[Example~11.4]{KKMMMsurvey}). 

\begin{definition}[{\cite[Section~2]{KKMM1}}] \label{definition_N-quasimorphism}
Let $\Gg$ be a group and $\Ng$ its normal subgroup.
\begin{enumerate}[label=(\arabic*)]
 \item A function $\psf \colon \Gg \to \RR$ is an \emph{$\Ng$-quasimorphism on $\Gg$} if $\DD''_{\Gg,\Ng}(\psf)$ is finite, where $\DD''_{\Gg,\Ng}(\psf)$ is defined to be
\[
\sup\left\{\max\{|\psf(\gl \xl) - \psf(\gl) - \psf(\xl)|,\, |\psf(\xl \gl) - \psf(\xl) - \psf(\gl)|\}\;|\; \gl\in \Gg,\ \xl\in \Ng\right\}.
\]
 \item An \emph{$\Ng$-homomorphism on $\Gg$} is an $\Ng$-quasimorphism $\psf \colon \Gg \to \RR$ on $\Gg$ such that $\DD''_{\Gg,\Ng}(\psf) = 0$.
  \item Define $\rQQQ_{\Ng}(\Gg)$ as the $\RR$-linear space of $\Ng$-quasimorphisms on $\Gg$, and $\HHH^1_{\Ng}(\Gg)$ as the $\RR$-linear space of $\Ng$-homomorphisms on $\Gg$.
\end{enumerate}
\end{definition}

We note that $\DD''_{\Gg,\Ng}$ induces a norm on $\rQQQ_{\Ng}(\Gg) / \HHH^1_{\Ng}(\Gg)$; we write $\DD''_{\Gg,\Ng}$ for this norm as well.

\begin{proposition}[{\cite[Proposition~2.4]{KKMM1}}] \label{prop 3.4.1}
For every $\muf\in \QQQ(\Ng)^{\Gg}$, there exists $\psf \in \rQQQ_{\Ng}(\Gg)$ such that  $\psf |_\Ng = \muf$ and that $\DD''_{\Gg, \Ng}(\psf) = \DD_{\Ng}(\muf)$.
\end{proposition}

Recall from Definition~\ref{defn=chainspace} that the $\RR$-linear space $\mathcal{C}_{\RR}(\Gg;\Ng)$ is defined as 
\[
\mathcal{C}_{\RR}(\Gg,\Ng)=\CCC_1(\Ng;\RR)\cap \BBB'_1(\Gg,\Ng;\RR)=\CCC_1(\Ng;\RR)\cap \partial \CCC'_2(\Gg,\Ng;\RR).
\]
If we set $\ZZZ'_2(\Gg,\Ng;\RR)=\Ker(\partial|_{\CCC'_2(\Gg,\Ng;\RR)})$, then $\partial|_{\CCC'_2(\Gg,\Ng;\RR)}$ induces an $\RR$-linear space isomorphism 
\begin{equation}\label{eq=BBB'}
\BBB'_1(\Gg,\Ng;\RR)\cong \CCC'_2(\Gg,\Ng;\RR)/\ZZZ'_2(\Gg,\Ng;\RR).
\end{equation}
Now, we equip $\CCC'_2(\Gg,\Ng;\RR)$ with the $\ell^1$-norm $\|\cdot\|_1$. Thus, we endow $\BBB'_1(\Gg,\Ng;\RR)$ with the quotient norm by \eqref{eq=BBB'}, and we write $\|\cdot\|'$ for this norm. In other words, for $c\in \BBB'_1(\Gg,\Ng;\RR)$, $\| c\|' = \inf \{ \| c' \|_1 \; | \; c' \in \CCC_2'(\Gg,\Ng ; \RR),\ \partial c'=c \}$. 
We equip $\mathcal{C}_{\RR}(\Gg,\Ng)$ with this norm $\|\cdot\|'$ as well.

\begin{proposition}[{\cite[Proposition~3.5]{KKMM1}; see also \cite[Proposition~5.7]{KKMMM_KIAS}}]\label{prop=dualQNG}
The normed space $(\rQQQ_{\Ng}(\Gg) / \HHH^1_{\Ng}(\Gg), \DD''_{\Gg,\Ng})$ is isometrically isomorphic to the continuous dual of the space $(\BBB_1'(\Gg,\Ng; \RR), \| \cdot\|')$, and the isomorphism is induced by the  pairing
\begin{equation}\label{eq=B_1pairing}
\Big(\rQQQ_{\Ng}(\Gg) / \HHH^1_{\Ng}(\Gg)\Big)\times \BBB_1'(\Gg,\Ng; \RR)\to \RR;\quad ([\psf],c)\mapsto \psf(c).
\end{equation}
\end{proposition}

\subsection{Refined version of the generalized mixed Bavard duality theorem}\label{subsec=refinedBavard}
Here, we proceed to the study of the refined version of the generalized mixed Bavard duality theorem. By Proposition~\ref{prop=dualQNG}, it is natural to ask whether $(\QQQ(\Ng)^{\Gg}/\HHH^1(\Ng)^{\Gg},\DD_{\Ng})$ is the continuous dual of a certain space of chains. To answer this question, we introduce the following normed space. We note that Calegari \cite{Calegari} considered a quotient vector space similar to this for the case where $\Ng=\Gg$, but without taking the norm closure.

\begin{definition}[$\ChGN$]\label{defn=ChGN}
Let $\Gg$ be a group and $\Ng$ its normal subgroup. 
\begin{enumerate}[label=(\arabic*)]
  \item We define the $\RR$-linear subspace $\mathrm{h}(\Gg,\Ng)$  of $\mathcal{C}_{\RR}(\Gg,\Ng)$ to be the $\RR$-span of $\{\xl^k-k \xl \;|\; \xl \in \Ng,\, k\in \ZZ\}$. 
  \item We define
\[
\ChGN=\mathcal{C}_{\RR}(\Gg,\Ng)/\overline{\mathrm{h}(\Gg,\Ng)}^{\|\cdot\|'}.
\]
Here, for a subspace $U$ of $\mathcal{C}_{\RR}(\Gg,\Ng)$, the symbol $\overline{U}^{\|\cdot\|'}$ means the norm closure of $U$ in $\|\cdot\|'$. Define $\|\cdot\|_{\hhh}$ to be the quotient norm on $\ChGN$.
  \item For $R=\ZZ$ or $R=\QQ$, define $\mathcal{C}^{\hhh}_R(\Gg,\Ng)$ to be the image of $\mathcal{C}_R(\Gg,\Ng)$ under the canonical projection $\mathcal{C}_{\RR}(\Gg,\Ng)\twoheadrightarrow \ChGN$.
\end{enumerate}
We write  $\ChG$ and $\mathcal{C}^{\hhh}_R(\Gg)$ for $\mathcal{C}^{\hhh}(\Gg,\Gg)$ and  $\mathcal{C}^{\hhh}_R(\Gg,\Gg)$, respectively.  
\end{definition}

We note that $\mathcal{C}^{\hhh}_{\QQ}(\Gg,\Ng)$ is dense in $(\ChGN,\|\cdot\|_{\hhh})$ because $\mathcal{C}_{\QQ}(\Gg,\Ng)$ is dense in $(\mathcal{C}_{\RR}(\Gg,\Ng),\|\cdot\|')$. By abuse of notation, we express an element $[c]$ in $\ChGN$ as $c$, for $c\in \mathcal{C}_{\RR}(\Gg,\Ng)$. We note that for $\muf\in \QQQ(\Ng)^{\Gg}$ and $[c]\in \ChGN$, the value $\muf(c)$ is well-defined because $\muf$ is homogeneous.

\begin{proposition}\label{prop=dualQN}
The duality pairing \eqref{eq=B_1pairing} induces a well-defined pairing
\begin{equation}\label{eq=C^hpairng}
\Big(\QQQ(\Ng)^{\Gg} / \HHH^1(\Ng)^{\Gg},\DD_{\Ng}\Big)\times (\ChGN,\|\cdot\|_{\hhh})\to \RR;\quad ([\muf],c)\mapsto \muf(c),
\end{equation}
and this is the duality pairing for $(\ChGN,\|\cdot\|_{\hhh})$. In particular, $(\ChGN,\|\cdot\|_{\hhh})^{\ast}$ is isometrically isomorphic to $(\QQQ(\Ng)^{\Gg} / \HHH^1(\Ng)^{\Gg},\DD_{\Ng})$.
\end{proposition}

By Proposition~\ref{prop=FA}~(1), we in particular have 
\begin{equation}\label{eq=HBh}
\|c\|_{\hhh}=\sup_{[\muf]\in (\QQQ(\Ng)^{\Gg} / \HHH^1(\Ng)^{\Gg})\setminus \{0\}}\frac{|\muf(c)|}{\DD_{\Ng}(\muf)}
\end{equation}
for every $c\in \ChGN$.

\begin{proof}[Proof of Proposition~\textup{\ref{prop=dualQN}}]
Set $X=\rQQQ_{\Ng}(\Gg)/\HHH^1_{\Ng}(\Gg)$. First, we apply Lemma~\ref{lem=dualquotient}~(1) for a subspace $\mathcal{C}_{\RR}(\Gg,\Ng)$ of $\BBB'(\Gg,\Ng;\RR)$. Then, by Proposition~\ref{prop=dualQNG}, 
\begin{equation}\label{eq=dual/C}
(X/\mathcal{C}_{\RR}(\Gg,\Ng)^{\perp},\|\cdot\|)\times (\mathcal{C}_{\RR}(\Gg,\Ng),\|\cdot\|')\to \mathbb{R};\quad ([\psf],c)\mapsto \psf(c)
\end{equation}
is the duality pairing for $(\mathcal{C}_{\RR}(\Gg,\Ng),\|\cdot\|')$. Here, $\|\cdot\|$ is the quotient norm on $X/\mathcal{C}_{\RR}(\Gg,\Ng)^{\perp}$, namely, for $\psf\in \rQQQ_{\Ng}(\Gg)$,
\[
\|[\psf]\|=\inf \{\DD''_{\Gg,\Ng}(\psf+f)\;|\; f+\HHH^1_{\Ng}(\Gg)\in \mathcal{C}_{\RR}(\Gg,\Ng)^{\perp}\}.
\]
For $\psf\in \rQQQ_{\Ng}(\Gg)$, we claim that 
\begin{equation}\label{eq=defectabove}
\|[\psf]\|\geq  \DD_{\Ng}(\psf|_{\Ng}).
\end{equation}
Indeed, let $f\in \rQQQ_{\Ng}(\Gg)$ with $f|_{\mathcal{C}_{\RR}(\Gg,\Ng)}\equiv 0$. Then, by the definitions of $\DD''_{\Gg,\Ng}$ and $\DD_{\Ng}$, we have
\begin{align*}
\DD''_{\Gg,\Ng}(\psf+f)&\geq \DD_{\Ng}((\psf+f)|_{\Ng})\\
&=\sup\{|(\psf+f)(\xl_1\xl_2-\xl_1-\xl_2)|\;|\;\xl_1,\xl_2\in \Ng\}\\
&=\sup\{|\psf(\xl_1\xl_2-\xl_1-\xl_2)|\;|\;\xl_1,\xl_2\in \Ng\}=\DD_{\Ng}(\psf|_{\Ng}).
\end{align*}
Here, note that $\xl_1\xl_2-\xl_1-\xl_2 \in \mathcal{C}_{\RR}(\Gg,\Ng)$ for all $\xl_1,\xl_2\in \Ng$.

Secondly, apply Lemma~\ref{lem=dualquotient}~(2) for the quotient space $\ChGN$ of $\mathcal{C}_{\RR}(\Gg,\Ng)$. Then, by \eqref{eq=dual/C}, the pairing
\begin{equation}\label{eq=dualhperp}
\Big(\big(\overline{\mathrm{h}(\Gg,\Ng)}^{\|\cdot\|'}\big)^{\perp},\|\cdot\|\Big)\times (\ChGN,\|\cdot\|_{\hhh})\to \mathbb{R};\quad ([\psf],c)\mapsto \psf(c)
\end{equation}
is the duality pairing for $(\ChGN,\|\cdot\|_{\hhh})$. Note that $\big(\overline{\mathrm{h}(\Gg,\Ng)}^{\|\cdot\|'}\big)^{\perp}$ equals $\mathrm{h}(\Gg,\Ng)^{\perp}$ by continuity. Let $[\psf]\in \mathrm{h}(\Gg,\Ng)^{\perp} \subseteq X/\mathcal{C}_{\RR}(\Gg,\Ng)^{\perp}$. Then, we  first  claim that $\psf|_{\Ng}\in \QQQ(\Ng)^{\Gg}$. Indeed, since $[\psf]\in \mathrm{h}(\Gg,\Ng)^{\perp}$, for every $\xl\in \Ng$ and $k\in \ZZ$ we have
\[
\psf|_{\Ng}(\xl^k)-k\cdot \psf|_{\Ng}(\xl)=0.
\]
In other words, $\psf|_{\Ng}$ is homogeneous. The map $\psf|_{\Ng}$ is a quasimorphism on $\Ng$ by \eqref{eq=defectabove}. To see that $\psf|_{\Ng}$ is $\Gg$-invariant, for every $\gl\in \Gg$ and $\xl\in \Ng$ we have
\[
\psf(\gl \xl \gl^{-1}) + \psf(\gl ) \sim_{D''_{\Gg,\Ng}(\psf)} \psf(\gl \xl) \sim_{D''_{\Gg,\Ng}(\psf)} \psf(\gl ) + \psf(\xl).
\]
Hence, we have
\begin{equation}\label{eq=quasiinv}
\psf(\gl \xl \gl^{-1}) \sim_{2D''_{\Gg,\Ng}(\psf)} \psf(\xl).
\end{equation}
By homogeneity of $\psf|_{\Ng}$, Lemma~\ref{lem=quasiinv} implies that $\psf|_{\Ng}\in \QQQ(\Ng)^{\Gg}$.  Our second claim is  that $\|[\psf]\|=\DD_{\Ng}(\psf|_{\Ng})$. 
To see this, we will prove that
\begin{equation}\label{eq=defectbelow}
\|[\psf]\|\leq  \DD_{\Ng}(\psf|_{\Ng}).
\end{equation}
Let $\muf=\psf|_{\Ng}$. Then, by the  first  claim, $\muf\in \QQQ(\Ng)^{\Gg}$. By Proposition~\ref{prop 3.4.1}, there exists $\hat{\psf}\in \rQQQ_{\Ng}(\Gg)$ such that $\hat{\psf}|_{\Ng}=\muf$ and that $\DD''_{\Gg,\Ng}(\hat{\psf})=\DD_{\Ng}(\muf)$. Then, $(\hat{\psf}-\psf)|_{\Ng}\equiv 0$ and in particular, $\hat{\psf}-\psf$ vanishes on $\mathcal{C}_{\RR}(\Gg,\Ng)$. Therefore, we have
\[
\|[\psf]\|\leq \DD''_{\Gg,\Ng}(\psf +(\hat{\psf}-\psf))= \DD''_{\Gg,\Ng}(\hat{\psf})=\DD_{\Ng}(\muf),
\]
as desired. Now, \eqref{eq=defectbelow} and \eqref{eq=defectabove}  complete  the proof of the  second  claim.

By the two claims above and Proposition~\ref{prop 3.4.1}, we conclude that the map $\psf\mapsto \psf|_{\Ng}$ induces a surjective linear isometry $(\mathrm{h}(\Gg,\Ng)^{\perp},\|\cdot\|)\stackrel{\cong}{\to} (\QQQ(\Ng)^{\Gg},\DD_{\Ng})$. 
By rewriting the pairing \eqref{eq=dualhperp} via this isometry, we obtain the desired duality pairing \eqref{eq=C^hpairng}.
\end{proof}

Now we are in position to state the refined version of the generalized mixed Bavard duality theorem.

\begin{theorem}[refined version of the generalized mixed Bavard duality theorem]\label{thm=refinedgmBavard}
Let $\Gg$ be a group and $\Ng$ its normal subgroup. Then, the following hold.
\begin{enumerate}[label=\textup{(\arabic*)}]
  \item For every $c\in \mathcal{C}_{\QQ}(\Gg,\Ng)$, we have
\[
\scl_{\Gg,\Ng}(c)=\frac{1}{2}\|c\|_{\hhh}.
\]
  \item The function $\scl_{\Gg,\Ng}\colon \mathcal{C}^{\hhh}_{\QQ}(\Gg,\Ng)\to \RR_{\geq 0}$ admits a unique continuous extension
\[
\scl_{\Gg,\Ng}\colon \ChGN\to \RR_{\geq 0}.
\]
Furthermore, the map
\begin{equation*}
\Big(\QQQ(\Ng)^{\Gg} / \HHH^1(\Ng)^{\Gg},2\DD_{\Ng}\Big)\times (\ChGN,\scl_{\Gg,\Ng})\to \RR;\quad ([\muf],c)\mapsto \muf(c),
\end{equation*}
is the duality pairing for $(\ChGN,\scl_{\Gg,\Ng})$. In particular, $(\ChGN,\scl_{\Gg,\Ng})^{\ast}$ is isometrically isomorphic to $(\QQQ(\Ng)^{\Gg} / \HHH^1(\Ng)^{\Gg},2\DD_{\Ng})$.
\end{enumerate}
\end{theorem}

\begin{proof}
By combining Theorem~\ref{thm=gmBavard} with \eqref{eq=HBh}, we obtain (1). For the former assertion of (2), recall that $\mathcal{C}^{\hhh}_{\QQ}(\Gg,\Ng)$ is dense in $(\ChGN,\|\cdot\|_{\hhh})$. Hence, the unique continuous extension $\scl_{\Gg,\Ng}\colon \ChGN\to \RR_{\geq 0}$ satisfies 
\begin{equation}\label{eq=choudo1/2}
\scl_{\Gg,\Ng}\equiv\frac{1}{2}\|\cdot\|_{\hhh}.
\end{equation}
Now, the latter assertion of (2) immediately follows from \eqref{eq=choudo1/2} and Proposition~\ref{prop=dualQN}. This  completes the  proof.
\end{proof}

We remark that Proposition~\ref{prop=dualQN} immediately provides the following corollary, which was previously obtained in \cite[Theorem~7.4]{KKMMM}.

\begin{corollary}\label{cor=QNBanach}
Let $\Gg$ be a group and $\Ng$ its normal subgroup. Then, $(\QQQ(\Ng)^{\Gg}/\HHH^1(\Ng)^{\Gg},\DD_{\Ng})$ is a Banach space.
\end{corollary}

\begin{proof}
By Proposition~\ref{prop=dualQN}, this space is the continuous dual of a real normed space,  and hence is a Banach space (see \cite[Theorem 4.1]{Rudin}). 
\end{proof}

Let $\Gg$ be a group and $\Ng$ its normal subgroup. Recall from Subsection~\ref{subsec=resultsCK} that $i\colon \Ng\hookrightarrow \Gg$ induces an $\RR$-linear map
\[
\widehat{i^{\ast}}\colon (\QQQ(\Gg)/\HHH^1(\Gg),\DD_{\Gg})\to (\QQQ(\Ng)^{\Gg}/\HHH^1(\Ng)^{\Gg},\DD_{\Ng});\quad [\phf]\mapsto [\phf|_{\Ng}],
\]
and that the quotient $\RR$-linear space $\WGN$, defined in Definition~\ref{defn=W}, can be regarded as the (algebraic) cokernel of $\widehat{i^{\ast}}$.

\begin{proposition}\label{prop=findim}
If $\dim_{\RR}\WGN$ is finite, then $\widehat{i^{\ast}}(\QQQ(\Gg)/\HHH^1(\Gg))$ is norm-closed in $(\QQQ(\Ng)^{\Gg}/\HHH^1(\Ng)^{\Gg},\DD_{\Ng})$.
\end{proposition}

\begin{proof}
Apply Corollary~\ref{cor=QNBanach} for the case where $(\Gg,\Ng)=(\Gg,\Gg)$ and $(\Gg,\Ng)=(\Gg,\Ng)$, respectively. Then, $(\QQQ(\Gg)/\HHH^1(\Gg),\DD_{\Gg})$ and $(\QQQ(\Ng)^{\Gg}/\HHH^1(\Ng)^{\Gg},\DD_{\Ng})$ are both Banach spaces. Now, Corollary~\ref{cor=findim} ends our proof.
\end{proof}

\begin{remark}\label{rem=coarsehom}
As we discussed in Subsection~\ref{subsec=coarsekernel}, a genuine bi-invariant metric $d^+_{\scl_{\Gg,\Ng}}$ is defined on $\CGN$ by \eqref{eq=d+scl}. Since $(\CGN,d^+_{\scl_{\Gg,\Ng}})$ is coarsely equivalent to $(\CGN,d_{\scl_{\Gg,\Ng}})$, we may regard $(\CGN,d_{\scl_{\Gg,\Ng}})$ as a coarse group. By Lemma~\ref{lem=CGNchain} we can define a map
\[
(\CGN,\;\cdot\;,d_{\scl_{\Gg,\Ng}})\to (\mathcal{C}^{\hhh}(\Gg,\Ng),+,\sclGN);\quad \yl\mapsto \yl,
\]
which is not a genuine group homomorphism in general. Nevertheless, this is a coarse homomorphism. Indeed, by Theorem~\ref{thm=gmBavard}, for all $\yl_1,\yl_2\in \Ng$ we have
\[
\sclGN(\yl_1+\yl_2-\yl_1\cdot \yl_2)\leq\frac{1}{2}.
\]
Since this map has the trivial (bounded) coarse kernel, in this manner we may regard $(\CGN,d_{\scl_{\Gg,\Ng}})$ as a coarse subgroup of $(\mathcal{C}^{\hhh}(\Gg,\Ng),\sclGN)$.
\end{remark}

\section{Proofs of the main theorems}\label{sec=proofmain}
In this section, we prove our main theorems. First, we show Theorem~\ref{mthm=main_Qcoarsekernel}; then we proceed to the proof of Theorem~\ref{mthm=main_biLip}. In addition, we prove Theorem~\ref{thm=cCKextend}.

The basic strategy here is to apply functional analysis on the spaces of chains. More precisely, we employ Banach's closed range theorem (Theorem~\ref{thm=Banach}). The refined version of the generalized mixed Bavard duality theorem, Theorem~\ref{thm=refinedgmBavard}, plays a key role here. Recall from Remark~\ref{rem=complete} that the assumption of completeness of $X$ and $Y$ is essential in Theorem~\ref{thm=Banach}. Hence, the quotient normed space $(\ChGN,\sclGN)$ appearing in Theorem~\ref{thm=refinedgmBavard} itself may \emph{not} be suited for our purpose, because $\mathcal{C}_{\RR}(\Gg,\Ng)$ is a subspace of $\CCC_1(\Ng;\RR)$, which only deals with finite sums. For this reason, we enlarge the normed space $(\ChGN,\sclGN)$ by taking its completion $(\barChGN,\barsclGN)$, so that Theorem~\ref{thm=Banach} applies. The price to pay here is the fact that $\barChGN$ is a huge and mysterious space: we will discuss this point in Remark~\ref{rem=ghost}. Nevertheless, in the setting of coarse geometry and coarse groups, $\ChGN$ is dense in $\barChGN$ and so $\ChGN$ is asymptotic to $\barChGN$ (recall Definition~\ref{defn=asymptotic}). Roughly speaking, through this process, the kernel of a bounded linear operator at the level of $\barChGN$ is approximated to be a \emph{coarse kernel} at the level of $\mathcal{C}^{\hhh}_{\QQ}(\Gg,\Ng)$; compare with  Theorem~\ref{prop=CKBanach}.  

This is the rough strategy for the proof of Theorem~\ref{mthm=main_Qcoarsekernel}. To obtain Theorem~\ref{mthm=main_biLip}, we are requested to take an approximation at the level of $\CGN$. At the moment, we are able to obtain a sequence $(\yl_n)_{n\in \NN}$ that witnesses the non-bi-Lipschitz equivalence between $\sclGN$ and $\sclG$ when $\WGN\ne 0$; see also Example~\ref{exa=iotakernel}.

\subsection{Proof of Theorem~\ref{mthm=main_Qcoarsekernel}}
The inclusion map $i\colon \Ng\hookrightarrow \Gg$ induces an $\RR$-linear map
\begin{equation}\label{eq=Riota}
\iota_{\RR}\colon (\mathcal{C}_{\RR}(\Gg,\Ng),\sclGN)\to (\mathcal{C}_{\RR}(\Gg),\sclG);\quad c\mapsto c.
\end{equation}
In relation to Theorem~\ref{thm=Banach}, a better point of view for \eqref{eq=Riota} may be to go to the level of $\ChGN$, because $\sclGN$ is a genuine norm at this level.  To summarize,  from $i\colon \Ng\hookrightarrow \Gg$ we consider
\[
\iota=\iota_{\RR}^{\hhh}\colon (\ChGN,\sclGN)\to (\ChG,\sclG);\quad [c]\mapsto  [\iota_{\RR}(c)]  =[c].
\]
Then, $\|\iota\|_{\mathrm{op}}\leq 1$. The restriction of $\iota$ to $\mathcal{C}^{\hhh}_{\QQ}$ gives rise to a map $\iota_{\QQ}^{\hhh}\colon (\mathcal{C}^{\hhh}_{\QQ}(\Gg,\Ng),\sclGN)\to (\mathcal{C}^{\hhh}_{\QQ}(\Gg),\sclG)$; 
$\iota_{\QQ}^{\hhh}$ can be seen as a coarse homomorphism between additive groups (compare with Example~\ref{exa=coarsegroup}). Actually, Theorem~\ref{mthm=main_Qcoarsekernel} is a statement on a coarse kernel of $\iota_{\QQ}^{\hhh}$. 

As is mentioned at the beginning of this section, the key procedure to our proof is to take the completions of $(\ChGN,\sclGN)$ and $(\ChG,\sclG)$, respectively. 
Write $(\barChGN,\barsclGN)$ and $(\barChG,\barsclG)$ for the completions of $(\ChGN,\sclGN)$ and $(\ChG,\sclG)$, respectively. Then, $\iota=\iota_{\RR}^{\hhh}$ admits a unique continuous extension
\[
\overline{\iota}\colon (\barChGN,\barsclGN)\to (\barChG,\barsclG),
\]
which satisfies $\|\overline{\iota}\|_{\mathrm{op}}\leq 1$. The duality pairing in Theorem~\ref{thm=refinedgmBavard}~(2) admits a unique continuous extension
\begin{equation}\label{eq=barduality}
\Big(\QQQ(\Ng)^{\Gg} / \HHH^1(\Ng)^{\Gg},2\DD_{\Ng}\Big)\times \big(\barChGN,\barsclGN\big)\to \RR;\quad ([\muf],\alpha)\mapsto \muf(\alpha),
\end{equation}
which is the duality pairing for $(\barChGN,\barsclGN)$. In particular, together with Proposition~\ref{prop=FA}~(1), for every $\alpha\in \barChGN$ we have
\begin{equation}\label{eq=barBavard}
\barsclGN(\alpha)=\sup_{[\muf]\in (\QNG/\HHH^1(\Ng)^{\Gg})\setminus \{0\}}\frac{|\muf(\alpha)|}{2\DD_{\Ng}(\muf)}.
\end{equation}

\begin{proof}[Proof of Theorem~\ref{mthm=main_Qcoarsekernel}]
By Theorem~\ref{thm=refinedgmBavard} (and \eqref{eq=barduality}), we have
\[
(\barChGN,\barsclGN)^{\ast}\cong (\QNG/\HHH^1(\Ng)^{\Gg},2\DD_{\Ng}).
\]
Considering the case where $\Ng=\Gg$, we have $(\barChG,\barsclG)^{\ast}\cong (\QQQ(\Gg)/\HHH^1(\Gg),2\DD_{\Gg})$ as well. 
Recall from Subsection~\ref{subsec=resultsCK} that $i\colon \Ng\hookrightarrow \Gg$ also induces an $\RR$-linear map
\[
\widehat{i^{\ast}}\colon \QQQ(\Gg)/\HHH^1(\Gg)\to \QNG/\HHH^1(\Ng)^{\Gg};\quad [\phi]\mapsto [\phi|_{\Ng}],
\]
and $\WGN$ is the (algebraic) cokernel of $\widehat{i^{\ast}}$.
The setup up to this point is illustrated in the following diagram, where the vertical arrows indicate the duality. Here, the equality $\widehat{i^{\ast}}=(\overline{\iota})^{\dagger}$ will be checked below.
\begin{equation}\label{eq=diagramscl}
    \begin{tikzpicture}[auto, baseline=(current bounding box.center)]
    \node[align=center] (01) at (0, 1.5) { $(\barChGN,\barsclGN)$ };
    \node[align=center] (11) at (5, 1.5) { $(\barChG,\barsclG)$ };
    \node[align=center] (00) at (0, 0) { $(\QNG/\HHH^1(\Ng)^{\Gg},2\DD_{\Ng})$ };
    \node[align=center] (10) at (5, 0) { $(\QQQ(\Gg)/\HHH^1(\Gg),2\DD_{\Gg})$ };
    \draw[->, decorate, decoration={snake}] (01) to node { } (00);
    \draw[->] (01) to node { $\scriptstyle \overline{\iota}$ } (11);
    \draw[->, decorate, decoration={snake}] (11) to node { } (10);
    \draw[<-] (00) to node { $\scriptstyle \widehat{i^{\ast}}=(\overline{\iota})^{\dagger}$ } (10);
    \end{tikzpicture}
\end{equation}

Note that $\mathcal{C}^{\hhh}_{\QQ}(\Gg,\Ng)$ is dense in $(\barChGN,\barsclGN)$ because it is dense in $(\ChGN,\sclGN)$. In particular, $\mathcal{C}^{\hhh}_{\QQ}(\Gg,\Ng)\asymp \barChGN$ in $(\barChGN,\barsclGN)$. Hence, $(\mathcal{C}^{\hhh}_{\QQ}(\Gg,\Ng),\sclGN)$ and $(\barChGN,\barsclGN)$ are isomorphic as coarse groups, and $(\mathcal{C}^{\hhh}_{\QQ}(\Gg),\sclG)$ and $(\barChG,\barsclG)$ are also isomorphic; a coarse kernel of $\iota_{\QQ}^{\hhh}$ is isomorphic to that of $\overline{\iota}$, under the existence of (at least one of) them. More concretely, suppose that a coarse kernel $\Omega$ of $\overline{\iota}$ exists. Then, since $\mathcal{C}^{\hhh}_{\QQ}(\Gg,\Ng)$ is dense in $(\barChGN,\barsclGN)$, there exists $A\subseteq \mathcal{C}^{\hhh}_{\QQ}(\Gg,\Ng)$ such that $A\asymp \Omega$ in $(\barChGN,\barsclGN)$. This $A$ is a coarse kernel of $\iota_{\QQ}^{\hhh}$. Thus, the problem reduces to study a coarse kernel of $\overline{\iota}$. 

We now apply  Theorem~\ref{prop=CKBanach}  to $\overline{\iota}$, and obtain the following.
\begin{enumerate}[label=(\alph*)]
  \item A coarse kernel of $\overline{\iota}$ exists if and only if $\overline{\iota}(\barChGN)$ is norm-closed in $(\barChG,\barsclG)$.
  \item If one of the two (equivalent) conditions of (a) is fulfilled, then $\Ker(\overline{\iota})$ is a coarse kernel of $\overline{\iota}$.
\end{enumerate}

In what follows, we show that $(\overline{\iota})^{\dagger}=\widehat{i^{\ast}}$, which is claimed in diagram \eqref{eq=diagramscl}. This claim is equivalent to saying that $\iota^{\dagger}=\widehat{i^{\ast}}$, and we will show this equality. By Theorem~\ref{thm=refinedgmBavard},
\begin{align*}
\langle \cdot ,\cdot\rangle_{\Gg,\Ng}&\colon\Big(\QQQ(\Ng)^{\Gg} / \HHH^1(\Ng)^{\Gg},2\DD_{\Ng}\Big)\times (\ChGN,\scl_{\Gg,\Ng})\to \RR;\quad \langle [\muf],c\rangle_{\Gg,\Ng}=\muf(c),\\
\langle \cdot ,\cdot\rangle_{\Gg}&\colon\Big(\QQQ(\Gg) / \HHH^1(\Gg),2\DD_{\Gg}\Big)\times (\ChG,\scl_{\Gg})\to \RR;\quad \langle [\phf],c\rangle_{\Gg}=\phf(c)
\end{align*}
are both duality pairings. Then for every $[\phf]\in \QQQ(\Gg) / \HHH^1(\Gg)$ and every $c\in \mathcal{C}_{\RR}(\Gg,\Ng)$, viewed as $[c]\in \ChGN$, we have
\[
\langle \widehat{i^{\ast}}[\phf],[c]\rangle_{\Gg,\Ng}=\langle [\phf|_{\Ng}],[c]\rangle_{\Gg,\Ng}=\phf|_{\Ng}(c)
=\phf(c)=\langle [\phf],\iota[c]\rangle_{\Gg},
\]
as desired. Here, recall that $c\in \mathcal{C}_{\RR}(\Gg,\Ng)\subseteq \CCC_1(\Ng;\RR)$. 

Now, we employ Theorem~\ref{thm=Banach} for the setting of \eqref{eq=diagramscl} with $S=\overline{\iota}$. Then, we in particular have the following.
\begin{enumerate}
  \item[(c)] $\overline{\iota}(\barChGN)$ is norm-closed in $(\barChG,\barsclG)$ if and only if $\widehat{i^{\ast}}(\QQQ(\Gg)/\HHH^1(\Gg))$ is norm-closed in $(\QQQ(\Ng)^{\Gg}/\HHH^1(\Ng)^{\Gg},2\DD_{\Ng})$.
  \item[(d)] If one of the two (equivalent) conditions of (c) is fulfilled, then $\widehat{i^{\ast}}(\QQQ(\Gg)/\HHH^1(\Gg))= \Ker(\overline{\iota})^{\perp}$.
\end{enumerate}

By the discussion at the beginning of this proof, (a) and (c) end the proof of (1). Now, we proceed to the proof of (2), and  assume that $\widehat{i^{\ast}}(\QQQ(\Gg)/\HHH^1(\Gg))$ is norm-closed in $(\QQQ(\Ng)^{\Gg}/\HHH^1(\Ng)^{\Gg},2\DD_{\Ng})$. Then, by (d) we have $\widehat{i^{\ast}}(\QQQ(\Gg)/\HHH^1(\Gg))= \Ker(\overline{\iota})^{\perp}$.
Then, by setting $U=\Ker(\overline{\iota})$ in Lemma~\ref{lem=dualquotient}~(1) we obtain the isometric isomorphism
\begin{equation}\label{eq=kernelcokernel}
\WGN\cong \Ker(\overline{\iota})^{\ast}.
\end{equation}
Here, $\WGN=\Coker(\widehat{i^{\ast}})$ is equipped with the quotient norm coming from $(\QQQ(\Ng)^{\Gg}/\HHH^1(\Ng)^{\Gg},2\DD_{\Ng})$. 
To adjust this difference to our norm on $\WGN$ in the statement of Theorem~\ref{mthm=main_Qcoarsekernel}, set $V=(\Ker(\overline{\iota}),2\barsclGN)$. This completes our proof of the former assertion of (2). To see the latter assertion of (2), assume that $\WGN$ is finite-dimensional. Then, the condition of (c) is satisfied by Corollary~\ref{cor=findim}. Also, $V$ is finite-dimensional since $\Rdim\WGN<\infty$. In particular, $V$ is reflexive, and we have the isometric isomorphisms $V=V^{\ast\ast}\cong \WGN^{\ast}$, as desired.
\end{proof}

\subsection{Proof of Theorem~\ref{mthm=main_biLip}}
\relax Here, we prove Theorem~\ref{mthm=main_biLip}. The new task  unseen in the proof of Theorem~\ref{mthm=main_Qcoarsekernel} is to approximate elements in $\Ker(\overline{\iota})$ by group elements in $\CGN$. It is unclear whether the actural approximation is possible in general, but we can always take an approximation \emph{with scaling-up}, as in the following lemma.

\begin{lemma}[approximation with scaling-up]\label{lem=approximation_scaling}
Let $\alpha\in \barChGN$ and $\varepsilon \in \RR_{>0}$. Then, there exist $k\in \NN$ and $\yl\in \CGN$ such that $\barsclGN(k\alpha-\yl)\leq \varepsilon k$.
\end{lemma}

\begin{proof}
Since $\mathcal{C}^{\hhh}_{\QQ}(\Gg,\Ng)$ (Definition~\ref{defn=ChGN}) is dense in $(\barChGN,\barsclGN)$, we can take $c\in \mathcal{C}_{\QQ}(\Gg,\Ng)$ such that $\barsclGN(\alpha-c)\leq \varepsilon/2$. Since $c\in \mathcal{C}_{\QQ}(\Gg,\Ng)$, there exists $l\in \NN$ such that $lc\in \mathcal{C}_{\ZZ}(\Gg,\Ng)$. By Lemma~\ref{lem=CGNchain}, there exist $m,m'\in \ZZ_{\geq 0}$ and $\xl_1, \cdots, \xl_{m}, \xbr_1, \cdots, \xbr_{m'} \in \Ng$ such that
\[
lc = \xl_1 + \cdots + \xl_{m} - \xbr_1 - \cdots - \xbr_{m'}
\]
and $\xl_1\cdots \xl_{m} \xbr_1^{-1}\cdots \xbr_{m'}^{-1}\in \CGN$. Take $t \in \NN$ with $m+m'\leq \varepsilon lt$. Finally, set 
\[
k=lt\in \NN\quad \mathrm{and}\quad  \yl=\xl_1^{t}\cdots \xl_{m}^{t} \xbr_1^{-t}\cdots \xbr_{m'}^{-t}.
\]
Note that $\yl\in \CGN$ because $\Ng/\CGN$ is abelian. We claim that
\begin{equation}\label{eq=sclapprox}
\sclGN(kc-\yl)\leq (m+m')/2\leq \varepsilon k/2.
\end{equation}
Indeed, for every $\muf\in \QQQ(\Ng)^{\Gg}$, we have 
\[
\muf(kc)=\muf(t\xl_1 + \cdots + t\xl_{m} - t\xbr_1 - \cdots - t\xbr_{m'})\sim_{(m+m')\DD_{\Ng}(\muf)}\muf(\yl).
\]
Then, the generalized mixed Bavard duality theorem (Theorem~\ref{thm=gmBavard})  implies \eqref{eq=sclapprox}. By \eqref{eq=sclapprox}, we now have $\barsclGN(k\alpha -\yl)\leq k\cdot \barsclGN(\alpha -c)+\barsclGN(kc-\yl)\leq \varepsilon k$.
\end{proof}

\begin{proof}[Proof of Theorem~\ref{mthm=main_biLip}]
First we prove ``(ii) implies (i)'' in the current framework, despite the fact that this direction was already proved in \cite{KKMMM}. Assume that $\WGN=0$. Then, in particular, $\widehat{i^{\ast}}(\QQQ(\Gg)/\HHH^1(\Gg))$ is norm-closed in $(\QQQ(\Ng)^{\Gg}/\HHH^1(\Ng)^{\Gg},2\DD_{\Ng})$. Hence, Theorem~\ref{thm=Banach}, together with Lemma~\ref{lem=dualquotient}, implies  \eqref{eq=kernelcokernel}. In the current case of $\WGN=0$, this implies that $\Ker(\overline{\iota})=0$. 
Also, Theorem~\ref{thm=Banach} implies that $\overline{\iota}(\barChGN)$ is norm-closed in $(\barChGN,\barsclG)$. Regard $\overline{\iota}$ as a bounded linear operator from $(\barChGN,\barsclGN)$ to $(\overline{\iota}(\barChGN),\barsclG)$. Then $(\overline{\iota}(\barChGN),\barsclG)$ is a Banach space, and  this bounded operator $\overline{\iota}$  is injective. Therefore, Proposition~\ref{prop=FA} applies to   $\overline{\iota}$,  and there exists $C\in \RR_{\geq 0}$ such that for every $\alpha\in \barChGN$,
\[
\barsclGN(\alpha)\leq C\cdot \barsclG(\alpha).
\]
By recalling Lemma~\ref{lem=CGNchain}, we obtain (i).

Now, we proceed to the proof of the converse, which is the main part of this proof. As is mentioned in the introduction, we will prove that ``not (ii) implies not (i).'' Assume that $\WGN\ne 0$. We will divide our proof into two cases, depending on whether  $\widehat{i^{\ast}}(\QQQ(\Gg)/\HHH^1(\Gg))$ is norm-closed in $(\QQQ(\Ng)^{\Gg}/\HHH^1(\Ng)^{\Gg},2\DD_{\Ng})$.

First, we treat the case where $\widehat{i^{\ast}}(\QQQ(\Gg)/\HHH^1(\Gg))$ is norm-closed in $(\QQQ(\Ng)^{\Gg}/\HHH^1(\Ng)^{\Gg},2\DD_{\Ng})$. Then, we have \eqref{eq=kernelcokernel}. Since $\WGN\ne 0$, we in particular have $\Ker(\overline{\iota})\ne 0$. Therefore, we can take $\alpha\in \Ker(\overline{\iota})$ with $\barsclGN(\alpha)=1$. For $n\in \NN$, set $\alpha_n=n\alpha$. Then for every $n\in \NN$ we have
\begin{equation}\label{eq=alpha}
\barsclGN(\alpha_n)=n\quad \mathrm{and}\quad \barsclG(\alpha_n)=0.
\end{equation}
\relax Let $n\in \NN$. Apply Lemma~\ref{lem=approximation_scaling} for $\varepsilon=1$. Then, we have $k_n\in \NN$ and $\yl_n\in \CGN$ such that $\barsclGN(k_n\alpha-\yl_n)\leq k_n$. By \eqref{eq=alpha}, we have
\[
\sclGN(\yl_n)\geq (n-1)k_n\quad \mathrm{and}\quad \sclG(\yl_n)\leq k_n.
\]
Here, recall that $\barsclG\leq \barsclGN$ on $\barChGN$. Then, $(\yl_n)_{n\in \NN}$ witnesses ``not (i)'':
\begin{equation}\label{eq=limity_n}
\lim_{n\to \infty}\dfrac{\sclGN(\yl_n)}{\sclG(\yl_n)}=\infty.
\end{equation}

Secondly, we deal with the remaining case, where $\widehat{i^{\ast}}(\QQQ(\Gg)/\HHH^1(\Gg))$ is not norm-closed in $(\QQQ(\Ng)^{\Gg}/\HHH^1(\Ng)^{\Gg},2\DD_{\Ng})$. Then, by Theorem~\ref{thm=Banach},  $\overline{\iota}(\barChGN)$ is not norm-closed in $(\barChGN,\barsclG)$. In particular, there exists a sequence $(\alpha_n)_{n\in \NN}$ in $\barChGN$ such that for every $n\in \NN$,
\[
\barsclGN(\alpha_n)\geq n\quad \mathrm{and}\quad \barsclG(\alpha_n)=1;
\]
compare with the argument in the proof of  Theorem~\ref{prop=CKBanach}.   \relax Again by Lemma~\ref{lem=approximation_scaling}, we can construct a sequence $(\yl_n)_{n\in \NN}$ in $[\Gg,\Ng]$ out of $(\alpha_n)_{n\in \NN}$ such that \eqref{eq=limity_n} holds. Therefore, we have proved that (i) implies (ii), as desired.
\end{proof}

\subsection{Proof of Theorem~\ref{thm=cCKextend}}\label{subsec=extend}
\begin{proof}[Proof of Theorem~\textup{\ref{thm=cCKextend}}]
Recall diagram~\eqref{eq=diagramscl} from the proof of Theorem~\ref{mthm=main_Qcoarsekernel}. By Theorem~\ref{mthm=main_Qcoarsekernel}, the existence of the coarse kernel $A$ implies that $\widehat{i^{\ast}}(\QQQ(\Gg)/\HHH^1(\Gg))$ is norm-closed in $(\QQQ(\Ng)^{\Gg}/\HHH^1(\Ng)^{\Gg},2\DD_{\Ng})$. Then, by Theorem~\ref{thm=Banach}, we have
\[
\widehat{i^{\ast}}(\QQQ(\Gg)/\HHH^1(\Gg))=\Ker(\overline{\iota})^{\perp}.
\]
Therefore, (i) holds if and only if  $\muf(\Ker(\overline{\iota}))=\{0\}$. Since $A\asymp \Ker(\overline{\iota})$ in $(\barChGN,\barsclGN)$, this condition is  equivalent to (ii) by \eqref{eq=barBavard}. This completes the proof.
\end{proof}

\relax At the level of $\CGN$, we have a corresponding result to Theorem~\ref{thm=cCKextend} up to scaling factors, as follows.

\begin{theorem}\label{thm=cCKextendCGN}
Let $\Gg$ be a group and $\Ng$ its normal subgroup. Let $i\colon \Ng\hookrightarrow \Gg$ be the inclusion map. Assume that $\widehat{i^{\ast}}(\QQQ(\Gg)/\HHH^1(\Gg))$ is norm-closed in $(\QQQ(\Ng)^{\Gg}/\HHH^1(\Ng)^{\Gg},\DD_{\Ng})$. Then, there exist a subset $Y$ of $\CGN$ and a map $\lambda\colon Y\to \NN$ such that for every $\muf\in \QNG$, the following are equivalent.
\begin{enumerate}[label=\textup{(\roman*)}]
  \item $[\muf]=0$ in $\WGN$. 
  \item The set $\{\lambda(\yl)^{-1}\muf(\yl)\;|\;\yl \in Y\}$ is a bounded subset of $\RR$.
\end{enumerate}
\end{theorem}

\begin{proof}
By Lemma~\ref{lem=approximation_scaling}, we can define two maps $\Phi\colon \Ker(\overline{\iota})\to \CGN$ and $\Lambda\colon \Ker(\overline{\iota})\to \NN$ such that for every $\alpha\in \Ker(\overline{\iota})$, we have 
\begin{equation}\label{eq=Lambda}
\barsclGN(\Lambda(\alpha)\alpha - \Phi(\alpha))\leq \Lambda(\alpha). 
\end{equation}
Note that \eqref{eq=Lambda} is equivalent to the following assertion by \eqref{eq=barBavard}: for every $\muf\in \QNG$, we have
\begin{equation}\label{eq=Lambdamu}
\Lambda(\alpha)^{-1}\muf(\Phi(\alpha))\sim_{2\DD(\muf)} \muf(\alpha).
\end{equation}
We now set $Y$ and $\lambda$ as 
\[
Y=\Phi(\Ker(\overline{\iota}))\quad \mathrm{and}\quad \lambda\colon Y\to \NN;\ \yl\mapsto \min \{\Lambda(\alpha)\;|\;\alpha \in \Phi^{-1}(\{\yl\})\},
\]
respectively. In what follows, we will show that these $Y$ and $\lambda$ work. 

As we have seen in the proof of Theorem~\ref{thm=cCKextend}, under our assumption, condition (i) is equivalent to $\muf(\Ker(\overline{\iota}))=\{0\}$. Assume (i). For every $\yl\in Y$, there exists $\alpha\in \Ker(\overline{\iota})$ such that $\Phi(\alpha)=\yl$ and $\Lambda(\alpha)=\lambda(\yl)$. Then by assumption (i) and \eqref{eq=Lambdamu}, for such $\alpha$ we have
\[
\lambda(\yl)^{-1}\muf(\yl)=\Lambda(\alpha)^{-1}\muf(\Phi(\alpha))\sim_{2\DD_{\Ng}(\muf)} \muf(\alpha)=0.
\]
Hence, we have (ii).

Conversely, assume (ii). Then, the set $\{\Lambda(\alpha)^{-1}\muf(\Phi(\alpha))\;|\;\alpha\in \Ker(\overline{\iota})\}$ is bounded as well by the definition of $\lambda$. Together with \eqref{eq=Lambdamu}, we conclude that the set $\muf(\Ker(\overline{\iota}))$ is bounded. Hence, $\muf(\Ker(\overline{\iota}))=\{0\}$, and we have (i).
\end{proof}

\subsection{Example of an explicit non-zero element of $\Ker(\overline{\iota})$}\label{subsec=remarks}
In our proof of Theorem~\ref{mthm=main_biLip}, we have seen that $\Ker(\overline{\iota})\ne 0$ if $\overline{\iota}(\barChGN)$ is norm-closed in $(\barChG,\barsclG)$ and $\WGN\ne 0$. However, it is another problem to obtain some explicit expression of a non-zero element $\alpha$ in $\Ker(\overline{\iota})$. In this subsection, we exhibit an example of such an expression of $\alpha$ for a certain group pair $(\Gg,\Ng)=(\Hg,\Kg)$. Let $\overline{\BBB}'_1(\Gg,\Ng;\RR)$ be the completion of $(\BBB'_1(\Gg,\Ng;\RR),\|\cdot\|')$, and let $\overline{\mathcal{C}}_{\RR}(\Gg,\Ng)$ be the norm-closure of $\mathcal{C}_{\RR}(\Gg,\Ng)$ in $\overline{\BBB}'_1(\Gg,\Ng;\RR)$. Then, $\mathcal{C}_{\RR}(\Gg,\Ng)\twoheadrightarrow \mathcal{C}^{\hhh}(\Gg,\Ng)$ extends to $\overline{\mathcal{C}}_{\RR}(\Gg,\Ng)\to \overline{\mathcal{C}}^{\hhh}(\Gg,\Ng)$, which is not necessarily surjective. 

\begin{example}[explicit non-zero element in $\Ker(\overline{\iota})$]\label{exa=iotakernel}
Here, we consider the case of $(\Gg,\Ng)=(\Hg,\Kg)$, where 
\begin{equation}\label{eq=HK}
\Hg=\langle a,b \;|\;[a,b]^2=e_{\Hg}\rangle\quad \mathrm{and}\quad 
\Kg=[\Hg,\Hg]. 
\end{equation}
Define a sequence $(c_n)_{n\in \NN}$ by $c_n=[a,b^{2^n}]/2^n$ for $n\in \NN$. Then, $(c_n)_{n\in \NN}$ is a sequence in $(\mathcal{C}_{\RR}(\Hg,\Kg),\|\cdot\|')$ (compare with \cite[Corollary 4.5]{MMM}). We claim that $(c_n)_{n\in \NN}$ is a Cauchy sequence. Indeed, by Proposition~\ref{prop=dualQNG}, for every $c\in \BBB'(\Hg,\Kg;\RR)$ we have
\begin{equation}\label{eq=dualB'}
\|c\|'=\sup_{[\psf]\in (\rQQQ_{\Kg}(\Hg) / \HHH^1_{\Kg}(\Hg))\setminus \{0\}}\frac{|\psf(c)|}{\DD''_{\Hg,\Kg}(\psf)}.
\end{equation}
Let $n\in \NN$. Let $\psf\in \rQQQ_{\Kg}(\Hg)$. Then, since 
\[
[a,b^{2^{n+1}}]=[a,b^{2^n}b^{2^n}]=[a,b^{2^n}](b^{2^n}[a,b^{2^n}]b^{-2^n}),
\]
we have $|\psf(c_{n+1}-c_n)|\leq (3/2^{n+1})\cdot \DD''_{\Hg,\Kg}(\psf)$; recall \eqref{eq=quasiinv}. By \eqref{eq=dualB'}, we have $\|c_{n+1}-c_n\|'\leq 3/2^{n+1}$, hence proving the claim.

Let $\hat{\alpha}\in \overline{\mathcal{C}}_{\RR}(\Hg,\Kg)$ be the norm-limit of $(c_n)_{n\in \NN}$ and $\alpha$ the image of $\hat{\alpha}$ by the map $\overline{\mathcal{C}}_{\RR}(\Hg,\Kg)\to \overline{\mathcal{C}}^{\hhh}(\Hg,\Kg)$. Then, for $(\Gg,\Ng)=(\Hg,\Kg)$ as in \eqref{eq=HK}, we can prove that $\Rdim \Ker(\overline{\iota})=1$ and $\alpha \in \overline{\mathcal{C}}^{\hhh}(\Hg,\Kg)$ generates $\Ker(\overline{\iota})$. We will explain the theory behind this for the case where $\Ng=[\Gg,\Gg]$ in our Part~II paper \cite{KKMMM_partII}.
\end{example}

\begin{remark}\label{rem=ghost}
The elements $\hat{\alpha}$ and $\alpha$, appearing in Example~\ref{exa=iotakernel}, are mysterious in the following sense. First, $\hat{\alpha}$ is a `ghost element,' whose meaning will be described below. If we naively try to express $\hat{\alpha}$ as an `infinite sum' $\sum\limits_{\xl\in \Kg}t_{\xl}\xl$, then we claim that all $t_{\xl}\in \RR$ must be zero. Indeed, for $\xl\in \Kg$, let $\delta_{\xl}\colon \Hg\to \RR$ be the delta function at $\xl$. Then, $\delta_{\xl}\in \rQQQ_{\Kg}(\Hg)$. Hence, the value $\langle [\delta_{\xl}],\hat{\alpha}\rangle$ of the duality pairing in Proposition~\ref{prop=dualQNG} makes sense, which should be regarded as $t_{\xl}$. However, we have
\[
\langle [\delta_{\xl}],\hat{\alpha}\rangle=\lim_{n\to \infty}\frac{\delta_{\xl}([a,b^{2^n}])}{2^n}=0.
\]
Therefore, the only `possible' way to express $\hat{\alpha}$ as an infinite sum will be ``$\hat{\alpha}=\sum\limits_{\xl\in \Kg}0\cdot \xl$.'' Nevertheless, as we mentioned in Example~\ref{exa=iotakernel}, $\alpha\ne 0$. In particular, $\hat{\alpha}$ is a non-zero element! Such an element $\hat{\alpha}$ is sometimes called a \emph{ghost element} (we remark that the `coefficient $t_{\xl}\in \RR$' does not make sense at the level of $\overline{\mathcal{C}}^{\hhh}(\Hg,\Kg)$, because $\delta_{\xl}$ does not belong to $\QQQ(\Kg)^{\Hg}$). 

Secondly, the element $\alpha$, at the level of $\overline{\mathcal{C}}^{\hhh}(\Hg,\Kg)$, is again mysterious. This is because for every $\muf\in \QQQ(\Kg)^{\Hg}$ we have
\[
\muf(\alpha)=\lim_{n\to \infty}\frac{\muf([a,b^{2^n}])}{2^n};
\]
this $\alpha$ itself embodies the limiting process. 
In fact, this formula implies that $\alpha\in \Ker(\overline{\iota})$. 
These examples may indicate that care is needed for the treatments of the completions $\overline{\mathcal{C}}_{\RR}(\Gg,\Ng)$ and $\overline{\mathcal{C}}^{\hhh}(\Gg,\Ng)$ of $\mathcal{C}_{\RR}(\Gg,\Ng)$ and $\mathcal{C}^{\hhh}(\Gg,\Ng)$, respectively.
\end{remark}

\section{Applications}\label{sec=applications}
\relax For a group $\Gg$ and its normal subgroup $\Ng$, recall from Section~\ref{sec=proofmain} that $(\overline{\mathcal{C}}^{\hhh}(\Gg,\Ng),\barsclGN)$ denotes the completion of $(\mathcal{C}^{\hhh}(\Gg,\Ng),\sclGN)$. Also, recall that $\overline{\iota}\colon (\barChGN,\barsclGN)\to (\barChG,\barsclG)$  is the continuous extension of $\iota=\iota^{\hhh}_{\RR}\colon (\ChGN,\sclGN)\to (\ChG,\sclG)$.

\subsection{Crushing theorem}\label{subsec=crushing}

Recall from our notation that $\Rdim$ means the real dimension of an $\RR$-linear space, taking values in $\ZZ_{\geq 0}\cup \{\infty\}$.

\begin{theorem}[crushing theorem]\label{thm=crushing}
Let $\Gg$ be a group and $\Ng$ its normal subgroup. Let $\Hg$ be a group and $\Kg$ its normal subgroup. Let $i\colon \Ng\hookrightarrow \Gg$ be the inclusion map. Assume that $\widehat{i^{\ast}}(\QQQ(\Gg)/\HHH^1(\Gg))$ is norm-closed in $(\QQQ(\Ng)^{\Gg}/\HHH^1(\Ng)^{\Gg},\DD_{\Ng})$, and that $\WW(\Hg,\Kg)$ is finite-dimensional. Assume that
\[
\Rdim \WGN > \Rdim \WW(\Hg,\Kg).
\]
Let  $\varphi\colon \Gg\to \Hg$ be a group homomorphism such that $\varphi(\Ng)\subseteq \Kg$. Then there exists a sequence $(\yl_n)_{n\in \NN}$ in $\CGN$ satisfying 
\[
\lim_{n\to\infty}\frac{\scl_{\Gg,\Ng}(\yl_n)}{\scl_{\Gg}(\yl_n)}=\infty\quad \text{and}\quad \lim_{n\to\infty}\frac{\scl_{\Gg,\Ng}(\yl_n)}{\scl_{\Hg,\Kg}(\varphi(\yl_n))}=\infty.
\]
\end{theorem}

\begin{remark}\label{rem=crushing}
Under the assumptions of Theorem~\ref{thm=crushing}, we have more detailed pieces of information on $\QQ$-chains as follows: there exists $B\subseteq \CQ(\Gg,\Ng)$ that fulfills the following four conditions.
\begin{enumerate}[label=\textup{(\alph*)}]
    \item $(B,\scl_{\Gg})$ is bounded.
    \item $(\varphi(B),\scl_{\Hg,\Kg})$ is bounded.
    \item $(B,\scl_{\Gg,\Ng})$ is isomorphic, as a coarse group, to a real normed space of dimension at least $\Rdim \WGN - \Rdim \WW(\Hg,\Kg)$; in particular, $(B,\scl_{\Gg,\Ng})$ is unbounded.
    \item For $\muf\in \QQQ(\Ng)^{\Gg}$, there exist $\psf\in \QQQ(\Gg)$ and $\nuf\in \QQQ(\Kg)^{\Hg}$ with 
\[
\muf-\psf|_{\Ng}-\nuf\circ  (\varphi|_{\Ng})  \in \HHH^1(\Ng)^{\Gg}
\]
if and only if $\muf(B)$ is bounded.
\end{enumerate}
\end{remark}

\begin{proof}[Proof of Theorem~\textup{\ref{thm=crushing}}]
First, we will prove the assertion stated in Remark~\ref{rem=crushing}, and then show Theorem~\ref{thm=crushing}. For group pairs $(G_1,N_1)$ and $(G_2,N_2)$, where $N_i$ is a normal subgroup of $G_i$ for $i\in \{1,2\}$, we write $\sigma\colon (G_1,N_1)\to (G_2,N_2)$ if $\sigma\colon G_1\to G_2$ is a group homomorphism that satisfies  $\sigma(N_1)\subseteq N_2$. For $\sigma\colon (G_1,N_1)\to (G_2,N_2)$, this $\sigma$ induces the following two $\RR$-linear maps: $\mathcal{C}^{\hhh}(G_1,N_1)\to \mathcal{C}^{\hhh}(G_2,N_2)$; $c\mapsto \sigma(c)$ and $\QQQ(N_2)^{G_2}\to \QQQ(N_1)^{G_1}$; $\muf\mapsto \muf\circ  (\sigma|_{\Ng_1} )$. These two maps furthermore induce the following $\RR$-linear maps:
\begin{align*}
&\sigma_{\ast}\colon (\overline{\mathcal{C}}^{\hhh}(G_1,N_1),\overline{\scl}_{G_1,N_1})\to (\overline{\mathcal{C}}^{\hhh}(G_2,N_2),\overline{\scl}_{G_2,N_2})\quad \mathrm{and}\\
& \sigma^{\ast}\colon (\QQQ(N_2)^{G_2}/\HHH^1(N_2)^{G_2},2\DD_{N_2})\to (\QQQ(N_1)^{G_1}/\HHH^1(N_1)^{G_1},2\DD_{N_1}),
\end{align*}
respectively. 

Now, given $\varphi\colon \Gg\to \Hg$, we can regard $\varphi$ as $\varphi\colon (\Gg,\Ng)\to (\Hg,\Kg)$ and as $\varphi'\colon (\Gg,\Gg)\to (\Hg,\Hg)$, respectively. Then, both of the following diagrams commute. 
\[
\xymatrix{
\overline{\mathcal{C}}^{\hhh}(\Gg,\Ng) \ar[r]^{\overline{\iota}^{\Gg}_{\Ng}} \ar[d]_{\varphi_{\ast}} &  \overline{\mathcal{C}}^{\hhh}(\Gg) \ar[d]^{\varphi'_{\ast}} \\
\overline{\mathcal{C}}^{\hhh}(\Hg,\Kg) \ar[r]_{\overline{\iota}^{\Hg}_{\Kg}} & \overline{\mathcal{C}}^{\hhh}(\Hg) \ar@{}[lu]|{\circlearrowright}
}
\qquad \quad
\xymatrix{
\QQQ(\Gg)/\HHH^1(\Gg) \ar[r]^{\widehat{i^{\ast}}^{\Gg}_{\Ng}} \ar[d]_{\varphi'{}^{\ast}} & \QQQ(\Ng)^{\Gg}/\HHH^1(\Ng)^{\Gg} \ar[d]^{\varphi^{\ast}} \\
\QQQ(\Hg)/\HHH^1(\Hg) \ar[r]_{\widehat{i^{\ast}}^{\Hg}_{\Kg}} & \QQQ(\Kg)^{\Hg}/\HHH^1(\Kg)^{\Hg} \ar@{}[lu]|{\circlearrowright}
}
\]
Here we indicate the group pair for $\overline{\iota}$ and $\widehat{i^{\ast}}$, respectively;  for instance, $\bar{\iota}^{\Gg}_{\Ng}$ stands for $\bar{\iota}$ for the group pair $(\Gg,\Ng)$.  Therefore, $\varphi\colon \Gg\to \Hg$ induces two $\RR$-linear maps 
\[
S_{\varphi}\colon \Ker(\overline{\iota}^{\Gg}_{\Ng})\to \Ker(\overline{\iota}^{\Hg}_{\Kg})
\quad
\mathrm{and}
\quad
T^{\varphi}\colon \WW(\Hg,\Kg)\to \WGN,
\]
respectively. Here, the norms on $\WW(\Hg,\Kg)$ and $\WGN$ are the quotient norms coming from $2\DD_{\Kg}$ and $2\DD_{\Ng}$, respectively. 

By the assumption and Corollary~\ref{cor=findim}, we have $\Ker(\overline{\iota}_{\Gg,\Ng})^{\ast}\cong \WGN$ and $\Ker(\overline{\iota}_{\Hg,\Kg})^{\ast}\cong \WW(\Hg,\Kg)$; recall our argument to deduce \eqref{eq=kernelcokernel} in the proof of Theorem~\ref{mthm=main_Qcoarsekernel}. 
 We can also show that $S_{\varphi}^{\dagger}=T^{\varphi}$. The following diagram \eqref{eq=diagramscl_GNHK} illustrates our argument up to this point, where two vertical arrows indicate the duality.
\begin{equation}\label{eq=diagramscl_GNHK}
    \begin{tikzpicture}[auto, baseline=(current bounding box.center)]
    \node[align=center] (01) at (0, 1.5) { $\Ker(\overline{\iota}_{\Gg,\Ng})$ };
    \node[align=center] (11) at (3, 1.5) { $\Ker(\overline{\iota}_{\Hg,\Kg})$ };
    \node[align=center] (00) at (0, 0) { $\WGN$ };
    \node[align=center] (10) at (3, 0) { $\WW(\Hg,\Kg)$ };
    \draw[->, decorate, decoration={snake}] (01) to node { } (00);
    \draw[->] (01) to node { $\scriptstyle S_{\varphi}$ } (11);
    \draw[->, decorate, decoration={snake}] (11) to node { } (10);
    \draw[<-] (00) to node { $\scriptstyle T^{\varphi}=S_{\varphi}^{\dagger}$ } (10);
    \end{tikzpicture}
\end{equation}

Note that $T^{\varphi}(\WW(\Hg,\Kg))$ is norm-closed in $\WGN$ because $\WW(\Hg,\Kg)$ is finite-dimensional. Hence, in the setting of \eqref{eq=diagramscl_GNHK} with $S=S_{\varphi}$,  item (ii) of Theorem~\ref{thm=Banach} holds. Therefore, by Theorem~\ref{thm=Banach} and Lemma~\ref{lem=dualquotient}~(1), we have 
\[
T^{\varphi}(\WW(\Hg,\Kg))=\Ker(S_{\varphi})^{\perp}\quad \text{and}\quad\Coker(T^{\varphi})\cong \Ker(S_{\varphi})^{\ast}.
\]
If $\Rdim \WGN<\infty$, then this implies that $\Coker(T^{\varphi})^{\ast}\cong \Ker(S_{\varphi})$, and hence $\Rdim\Ker(S_{\varphi})\geq \Rdim\WGN-\Rdim \WW(\Hg,\Kg)$. If $\Rdim \WGN=\infty$, then in this case we have 
\[
\Rdim\Ker(S_{\varphi})=\infty= \Rdim\WGN-\Rdim \WW(\Hg,\Kg).
\]
Hence, we have $\Rdim\Ker(S_{\varphi})\geq \Rdim\WGN-\Rdim \WW(\Hg,\Kg)$ in both cases. 

At the level of $\QQ$-chains, we can approximate $\Ker(S_{\varphi})$ by a subset $B$ of $\mathcal{C}^{\hhh}_{\QQ}(\Gg,\Ng)$ in such a way that $B\asymp \Ker(S_{\varphi})$ in $(\overline{\mathcal{C}}^{\hhh}(\Gg,\Ng),\barsclGN)$. This $B$ satisfies conditions (a), (b), (c) and (d) in Remark~\ref{rem=crushing}; recall the proof of Theorem~\ref{thm=cCKextend}. At the level of group elements in $\CGN$, we can construct a sequence $(\yl_n)_{n\in \NN}$ in $\CGN$ that satisfies the two assertions in Theorem~\ref{thm=crushing} \relax with the aid of Lemma~\ref{lem=approximation_scaling}. 
\end{proof}

\subsection{Existence of norm-attaining elements}
Here we exhibit applications of the refined version of the generalized mixed Bavard duality theorem (Theorem~\ref{thm=refinedgmBavard} and the duality pairing \eqref{eq=barduality}). The first one is the existence of extremal quasimorphisms in the mixed setting. We note that Calegari \cite{Calegari} argued in the setting of $\Ng=\Gg$. In this subsection, let $\Gg$ be a group and $\Ng$ its normal subgroup; let $i\colon \Ng \hookrightarrow \Gg$ be the inclusion map. 

\begin{proposition}\label{prop=extremal}
Assume that $\QQQ(\Ng)^{\Gg}\ne \HHH^1(\Ng)^{\Gg}$. Then, for every $c\in \mathcal{C}_{\RR}(\Gg,\Ng)$, there exists $\muf\in \QQQ(\Ng)^{\Gg}\setminus \HHH^1(\Ng)^{\Gg}$ such that
\[
\sclGN(c)=\frac{|\muf(c)|}{2\DD_\Ng (\muf)}
\]
holds.
\end{proposition}

\begin{proof}
This immediately follows from the Hahn--Banach theorem and the duality in Theorem~\ref{thm=refinedgmBavard} (see, for instance,  \cite[Theorem~3.3]{Rudin}  for this deduction).
\end{proof}

If $\widehat{i^{\ast}}(\QQQ(\Gg)/\HHH^1(\Gg))$ is norm-closed in $(\QNG/\HHH^1(\Ng)^{\Gg},\DD_{\Ng})$, then $\WGN$ is naturally equipped with the quotient norm.  Atsushi  Katsuda asked  the fifth-named author   whether we can take a representative that attains this quotient norm. Here is the affirmative answer.

\begin{proposition}\label{prop=attain}
Assume that $\widehat{i^{\ast}}(\QQQ(\Gg)/\HHH^1(\Gg))$ is norm-closed in $(\QNG/\HHH^1(\Ng)^{\Gg},\DD_{\Ng})$. Then, for every $\muf\in \QQQ(\Ng)^{\Gg}$, there exists $\phf\in \QQQ(\Gg)$ that satisfies 
\[
\DD_{\Ng}(\muf-\phf|_{\Ng})=\inf\{\DD_{\Ng}(\muf-\psf|_{\Ng})\;|\;\psf\in \QQQ(\Gg)\}.
\]
\end{proposition}

For a real dual Banach space $(Y,\|\cdot\|_Y)$ of the form $(X,\|\cdot\|_X)^{\ast}$, recall that the \emph{weak$^{\ast}$-topology} on $Y$ (with respect to the duality pairing $\langle\cdot,\cdot\rangle_X\colon Y\times X\to \RR$) is the weakest topology in which the following family of linear functionals
\[
\big(\langle \cdot ,x\rangle_X\colon Y\to \RR;\ y\mapsto \langle y,x\rangle_X\big)_{x\in X}
\]
are all continuous. We use the following result: for a non-empty weak$^{\ast}$-closed subset $A$ of a dual Banach space $Y$ and for $y\in Y$, the distance between $y$ and $A$ is realized. To prove this, by weak$^{\ast}$-compactness of closed balls (Banach--Alaoglu theorem; see \cite[Theorem~3.15]{Rudin}), the intersection of $B^{Y}_R(y)\cap A$ over all $R$ with $\mathrm{dist}(y,A)+1\geq R>\mathrm{dist}(y,A)$ is non-empty; every element there realizes $\mathrm{dist}(y,A)$.

\begin{proof}[Proof of Proposition~\textup{\ref{prop=attain}}]
By Theorem~\ref{thm=refinedgmBavard} and the assumption, we can apply Theorem~\ref{thm=Banach} for $S=\overline{\iota}\colon (\barChGN,\barsclGN)\to (\barChG,\barsclG)$. Then, we in particular have $\widehat{i^{\ast}}(\QQQ(\Gg)/\HHH^1(\Gg))=\Ker(\overline{\iota})^{\perp}$. By this expression, the set $\widehat{i^{\ast}}(\QQQ(\Gg)/\HHH^1(\Gg))$ is weak${}^\ast$-closed in $(\QQQ(\Ng)^{\Gg}/\HHH^1(\Ng)^{\Gg},2\DD_{\Ng})$, where  the weak${}^\ast$-topology on $\QQQ(\Ng)^{\Gg}/\HHH^1(\Ng)^{\Gg}$ is given by \eqref{eq=barduality}. Now, we employ the aforementioned result on distance realizers for $Y=(\QQQ(\Ng)^{\Gg}/\HHH^1(\Ng)^{\Gg},2\DD_{\Ng})$, $A=\widehat{i^{\ast}}(\QQQ(\Gg)/\HHH^1(\Gg))$, and $y=[\muf]$.
\end{proof}

\section{Proof of Theorem~\ref{thm=closedrange}}\label{sec=closedrange}

The goal of this section is to show Theorem~\ref{thm=closedrange}. Recall that $\Ng$ is a normal subgroup of $\Gg$, and $\Gamma$ is the quotient group $\Gg/\Ng$. Let $i \colon \Ng \hookrightarrow \Gg$ be the inclusion and $\ppi \colon \Gg \to \Gamma$ be the quotient map. Also, recall from Theorem~\ref{H2b_Banach} that the second bounded cohomology $\HHH^2_b(G)$ is a Banach space with the respect to the norm $\|\cdot\|$.

Let $K$ be the kernel of $i^* \colon \HHH^2_b(\Gg) \to \HHH^2_b(\Ng)^{\Gg}$. Recall that the map $\delta \colon \QQQ(\Gg) \to \HHH^2_b(\Gg)$ is defined in Subsection~\ref{subsec=H_b}. Consider the following two conditions:
\begin{enumerate}
\item[(A)] The image of $i^* \colon \HHH^2_b(\Gg) \to \HHH^2_b(\Ng)^{\Gg}$ is norm-closed.
\item[(B)] The space $K\cap \delta\QQQ(G)$ is finite-codimensional in $K$.
\end{enumerate}

\begin{lemma} \label{lem=a->A}
If $\HHH^3_b(\Gam)$ is finite-dimensional, then condition {\rm(A)} is satisfied.
\end{lemma}
\begin{proof}
By Theorem~\ref{theorem_Monod_sequence}, we have an exact sequence
\[ 0 \to \HHH^2_b(\Gam) \xrightarrow{\ppi^*} \HHH^2_b(\Gg) \xrightarrow{i^*} \HHH^2_b(\Ng)^{\Gg} \to \HHH^3_b(\Gam).\]
Since $\HHH^3_b(\Gam)$ is finite-dimensional, the (algebraic) cokernel of $\ppi^* \colon \HHH^2_b(\Gam) \to \HHH^2_b(\Gg)$ is finite-dimensional. Hence, Corollary~\ref{cor=findim} and Theorem~\ref{H2b_Banach} imply that $i^* \HHH^2_b(\Gg)$ is norm-closed.
\end{proof}

\begin{lemma} \label{lem=b->B}
If one of the following conditions are satisfied, then condition {\rm (B)} is satisfied:
\begin{enumerate}
\item[$(1)$] The $($algebraic$)$ cokernel of the map $\delta \colon \QQQ(\Gam) \to \HHH^2_b(\Gam)$ is finite-dimensional.
\item[$(2)$] $\HHH^2(\Gam)$ is finite-dimensional.
\end{enumerate}
\end{lemma}
\begin{proof}
By Proposition~\ref{prop_Q_exact}, we have an exact sequence
\[ 0 \to \HHH^1(\Gam) \to \QQQ(\Gam) \xrightarrow{\delta} \HHH^2_b(\Gam) \xrightarrow{c_\Gam} \HHH^2(\Gam).\]
Hence, (2) implies (1). Thus it suffices to see that (1) implies (B).

Suppose (1). Consider the following commutative diagram:
\[ \xymatrix{
& \QQQ(\Gamma) \ar[r]^{\ppi^*} \ar[d]^{\delta} & \QQQ(G) \ar[r]^{i^*} \ar[d]^{\delta} & \QQQ(N)^G \ar[d]^{\delta} \\
0 \ar[r] & \HHH^2_b(\Gamma) \ar[r]^{\ppi^*} & \HHH^2_b(G) \ar[r]^{i^*}  \ar@{}[lu]|{\circlearrowright}& \HHH^2_b(N)^G \ar@{}[lu]|{\circlearrowright}
}\]
Then  the  lower horizontal sequence of this diagram is exact (see Theorem~\ref{theorem_Monod_sequence}). In particular, $K=\ppi^{\ast}(\HHH^2_b(\Gamma))$. Let  $K_0=\ppi^* \delta \QQQ(\Gamma)$. Then, $K_0$ is finite-codimensional in $K$ by assumption (1). We also have
\[ K_0 =  \ppi^* \delta \QQQ(\Gamma) = \delta \ppi^*\QQQ(\Gamma) \subseteq \delta \QQQ(G).\]
Hence condition (B) is satisfied.
\end{proof}

We employ the following lemma in functional analysis.

\begin{lemma}\label{lem=findimquotient}
Let $X$ be a real normed space. Let $U$ and $V$ be two norm-closed subspaces of $X$. Assume that $U\cap V$ is finite-codimensional in $V$. Then, the image of $U$ under the quotient map $X\twoheadrightarrow X/V$ is norm-closed.
\end{lemma}

\begin{proof}
By the definition of the quotient topology, the conclusion holds if and only if $U+V$ is a norm-closed subspace of $X$; this holds if and only if the image of $V$ under $X\twoheadrightarrow X/U$ is norm-closed. The last condition holds because this image is a finite-dimensional subspace of a normed space $X/U$.
\end{proof}

\begin{theorem} \label{thm=closedrange2}
If both the conditions {\rm (A)} and {\rm (B)} are satisfied, then the image of $\widehat{i^*} \colon  (\QQQ(G) / \HHH^1(G) ,\DD_{\Gg}) \to (\QQQ(\Ng)^G / \HHH^1(N)^G, \DD_{\Ng})$ is norm-closed.
\end{theorem}
\begin{proof}
By Lemma~\ref{lem=eqnorms} we have
\begin{align} \label{eq=eqnorms}
\| [\delta \mu]\| \le D_\Ng(\mu) \le 2 \| [\delta \mu]\|
\end{align}
for every $\mu \in \QQQ(N)^G$. Let $\delta' \colon \QQQ(N)^G / \HHH^1(N)^G \to \HHH^2_b(N)^G$ denote the injective map induced by $\delta \colon \QQQ(N)^G \to \HHH^2_b(N)^G$.

Regard $\delta'$ as a map from a Banach space $(\QQQ(N)^G / \HHH^1(N)^G,\DD_{\Ng})$ (Corollary~\ref{cor=QNBanach}) to a Banach space $(\HHH^2_b(\Ng)^{\Gg},\|\cdot\|)$ (Theorem~\ref{H2b_Banach}). Then, by  \eqref{eq=eqnorms}, the subspace $\delta (\QQQ(N)^G) = \delta'(\QQQ(N)^G/\HHH^1(N)^G)$ is norm-closed in $\HHH^2_b(N)^G$. In the same manner, we obtain the norm-closedness of  the image of $ \delta' \colon \QQQ(G) / \HHH^1(G) \to \HHH^2_b(G)$ in $\HHH^2_b(G)$.

By condition (B), Lemma~\ref{lem=findimquotient} implies that the image  $E$  of $\delta \QQQ(G) = \delta'(\QQQ(G) / \HHH^1(G))$ under the quotient map $\HHH^2_b(G)\twoheadrightarrow \HHH^2_b(G) / K$ is norm-closed. By condition (A), $i^{\ast}{}' \colon \HHH^2_b(G) / K \to \HHH^2_b(N)^G$ is a closed embedding. Then, $i^* \delta \QQQ(G)$ is norm-closed in $\HHH^2_b(N)^G$ because it is the image of the closed subset  $E$  under the closed embedding $i^{\ast}{}' \colon \HHH^2_b(G) / K \to \HHH^2_b(N)^G$. Therefore, since
\[ 
  i^* \delta \QQQ(G) = \delta i^* \QQQ(G) \subseteq \delta \QQQ(N)^G =  \delta'(\QQQ(N)^G / \HHH^1(N)^G),
\]
$\delta i^* \QQQ(G)$ is norm-closed in $\delta' (\QQQ(N)^G / \HHH^1(N)^G)$. Since $\delta$ is a closed embedding, we conclude that $\widehat{i^*} (\QQQ(G) / \HHH^1(G))$ is norm-closed in $(\QQQ(\Ng)^G / \HHH^1(N)^G, \DD_{\Ng})$.
\end{proof}

\begin{proof}[Proof of Theorem~\textup{\ref{thm=closedrange}}]
It follows from Theorem~\ref{thm=closedrange2}, together with  Lemmas~\ref{lem=a->A} and \ref{lem=b->B}. 
\end{proof}

\section*{Acknowledgement}
The fifth-named author is grateful to Atsushi Katsuda for his question related to Proposition~\ref{prop=attain}. 
The first-named author, the second-named author, the fourth-named author and the fifth-named author are partially supported by JSPS KAKENHI Grant Number JP25K06994, JP24K16921, JP23K12975, and JP26K06799, respectively.
The third-named author is partially supported by JSPS KAKENHI Grant Number JP23K12971 and Research Fellowship Promoting International Collaboration, The Mathematical Society of Japan.


\bibliography{reference}
\bibliographystyle{abbrv}

\end{document}